\documentclass[onefignum,onethmnum,onetabnum,leqno]{siamltex}
\pdfoutput=1

\pdfoutput=1
\usepackage{amsmath,amsfonts,amssymb}
\usepackage{graphicx,verbatim,mathtools}
\usepackage{listings,color}
\usepackage{tikz,setspace}
\usetikzlibrary{patterns}
\usetikzlibrary{positioning}
\usepackage{pgfplots}
\usepackage[only,llbracket,rrbracket]{stmaryrd}
\usepackage{subcaption}

\newcommand{\cb}[1]{{\color{black}#1}}

\newtheorem{remark}{Remark}
\newtheorem{property}{Property}

\renewcommand{\a}{\alpha}
\newcommand{\bo}[1]{\mathbf{#1}} 
\newcommand{\llb}{\llbracket}
\newcommand{\rrb}{\rrbracket}
\newcommand{\lla}{\left\lbrace}
\newcommand{\rra}{\right\rbrace}
\renewcommand{\d}{\mathrm{d}}
\newcommand{\eps}{\varepsilon}
\newcommand{\abs}[1]{\lvert#1\rvert} 
\newcommand{\Labs}[1]{\left\lvert#1\right\rvert} 
\newcommand{\tends}{\rightarrow}

\newcommand{\norm}[1]{\lVert#1\rVert} 
\newcommand{\p}{\partial}
\newcommand{\half}{\textstyle{\frac{1}{2}}}

\DeclareMathOperator{\diam}{diam}

\DeclareMathOperator{\card}{card}
\DeclareMathOperator{\Div}{div}
\DeclareMathOperator{\Dim}{dim}
\DeclareMathOperator{\Curl}{curl}
\DeclareMathOperator{\Trace}{Tr}
\DeclareMathOperator{\divT}{\Div_{\mathrm{T}}}
\DeclareMathOperator{\nablaT}{\nabla_{\mathrm{T}}}

\DeclareMathOperator*{\argmin}{argmin}
\newcommand{\eval}[2]{\left. #1\right|_{#2}}
\newcommand{\pair}[2]{\langle #1,#2 \rangle}

\newcommand{\R}{\mathbb{R}}

\newcommand{\calA}{\mathcal{A}}
\newcommand{\calF}{\mathcal{F}}
\newcommand{\calT}{\mathcal{T}}

\newcommand{\calP}{\mathcal{P}}

\newcommand{\calK}{\mathcal{K}}

\newcommand{\Ld}{\Lambda}
\newcommand{\Om}{\Omega}
\newcommand{\DO}{\partial\Om}

\newcommand{\Ob}{\overline{\Omega}}

\newcommand{\Jstab}{J_h}

\newcommand{\normh}[1]{\norm{#1}_{2,h}}
\newcommand{\absj}[1]{\abs{#1}_{\mathrm{J},h}}
\newcommand{\cmu}{c_{\mu}}
\newcommand{\ceta}{c_{\eta}}
\newcommand{\ch}{c_{\calT}}
\newcommand{\cp}{c_{\calP}}

\newcommand{\Vh}{V_{h,\bo p}}
\newcommand{\VH}{V_{H,\bo q}}
\newcommand{\calFhi}{\calF_h^i}
\newcommand{\calFhib}{\calF_h^{i,b}}
\newcommand{\Fg}{F_\gamma}

\renewcommand{\dim}{d}
\newcommand{\kPA}{\kappa(\bo P^{-1}\bo A)}

\title{Nonoverlapping domain decomposition preconditioners for discontinuous Galerkin approximations of Hamilton--Jacobi--Bellman equations}
\author{Iain~Smears\footnotemark[2]}

\begin{document}

\renewcommand{\thefootnote}{\fnsymbol{footnote}}
\footnotetext[2]{Inria Paris, 2 Rue Simone Iff, 75589, Paris, France, iain.smears@inria.fr}
\renewcommand{\thefootnote}{\arabic{footnote}}

\maketitle

\begin{abstract}
We analyse a class of nonoverlapping domain decomposition preconditioners for nonsymmetric linear systems arising from discontinuous Galerkin finite element approximation of fully nonlinear Hamilton--Jacobi--Bellman (HJB) partial differential equations.
\cb{These nonsymmetric linear systems are uniformly bounded and coercive with respect to a related symmetric bilinear form, that is associated to a matrix $\bo A$. In this work, we construct a nonoverlapping domain decomposition preconditioner $\bo P$, that is based on $\bo A$, and we then show that the effectiveness of the preconditioner for solving the} nonsymmetric problems can be studied in terms of the condition number $\kPA$. In particular, we establish the bound $\kPA \lesssim 1+ p^6 H^3 /q^3 h^3$, where $H$ and $h$ are respectively the coarse and fine mesh sizes, and $q$ and $p$ are respectively the coarse and fine mesh polynomial degrees.
This represents the first such result for this class of methods that explicitly accounts for the dependence of the condition number on $q$; \cb{our analysis is founded upon} an original optimal order approximation result between fine and coarse discontinuous finite element spaces.
Numerical experiments demonstrate the sharpness of this bound.
Although the preconditioners are not robust with respect to the polynomial degree, our bounds quantify the effect of the coarse and fine space polynomial degrees. Furthermore, we  show computationally that these methods are effective in practical applications to nonsymmetric, fully nonlinear HJB equations under $h$-refinement for moderate polynomial degrees.
\end{abstract}



\section{Introduction}\label{sec:introduction}
In \cite{Smears2013,Smears2014,Smears2015}, discontinuous Galerkin finite element methods (DGFEM) were introduced for the numerical solution of linear nondivergence form elliptic equations and fully nonlinear Hamilton--Jacobi--Bellman (HJB) equations with Cordes coefficients. In these applications, the appropriate norm on the finite element space is a broken $H^2$-norm with penalization of the jumps in values and in first derivatives across the faces of the mesh. As a result, it is typical for the condition number of the discrete problems to be of order $p^8/h^4$, where $h$ is the mesh size and $p$ is the polynomial degree.
The purpose of this work is to study the application of a commonly used class of nonoverlapping domain decomposition preconditioners to these problems.

Nonoverlapping domain decomposition methods, along with their overlapping counterparts, have been successfully developed for a range of applications of DGFEM by many authors \cite{Antonietti2007,Antonietti2008,Antonietti2011,Antonietti2009,Feng2001,Feng2005,Lasser2003}. In order to solve a problem on a fine mesh $\calT_h$, these methods combine a coarse space solver, defined on a coarse mesh $\calT_H$, with local fine mesh solvers, defined on a subdomain decomposition $\calT_S$ of the domain $\Om$. The discontinuous nature of the finite element space leads to a significant flexibility in the choice of the decomposition $\calT_S$, which can either be overlapping or nonoverlapping.
As explained in the above references, these preconditioners possess many advantages in terms of simplicity and applicability, as they allow very general choices of basis functions, nonmatching meshes and varying element shapes, and are naturally suited for parallelization. It has been pointed out by various authors, such as Lasser and Toselli in \cite[p.~1235]{Lasser2003}, that nonoverlapping methods feature reduced inter-subdomain communication burdens, thus representing an advantage in parallel computations.

For problems involving $H^1$-type norms, such as divergence form second-order elliptic PDE, nonoverlapping additive Schwarz preconditioners for $h$-version methods \cite{Feng2001} lead to condition numbers of order $1+H/h$, where $H$ is the coarse mesh size, while overlapping methods lead to a condition number of order $1+H/\delta$, where $\delta$ is the subdomain overlap. 
For problems in $H^2$-type norms such as the biharmonic equation, the $h$-version analysis \cite{Feng2005} leads to condition numbers of order $1+H^3/h^3$.
We remark that the analysis in these works leaves the polynomial degree implicit inside the generic constants. However an analysis that keeps track of all parameters is important in practice for determining their effect on the performance of the preconditioners, even if robustness of the condition number cannot be guaranteed.
Antonietti and Houston \cite{Antonietti2011} were the first to keep track of the dependence on the polynomial degrees for this class of preconditioners for problems in $H^1$-norms, where they showed a condition number bound of order $1+p^2 H/h$. However, their numerical experiments lead them to conjecture the improved bound of order $1+p^2 H/qh$, where $q$ is the coarse space polynomial degree. \cb{This conjecture was recently proved in \cite{Antonietti2016} using ideas first developed in this work.}

As can be seen from the theoretical analysis in the above references, the effectiveness of the preconditioner depends in an essential way on the approximation properties between the coarse and fine spaces.
In the analysis of $h$-version DGFEM, it is sufficient to consider low-order projection operators from the fine space to the coarse space; for example, coarse element mean-value projections are employed in \cite{Feng2001} and local first-order elliptic projections are used in \cite{Feng2005}. However, low-order projections lead to suboptimal bounds for the condition number with respect to polynomial degrees. This work resolves this suboptimality through an original optimal order approximation result between coarse and fine spaces.

There are further classes of preconditioners for $p$-version and $hp$-version methods for problems in $H^1$-norms that achieve condition numbers either independent or depending only polylogarithmically on the polynomial degree, such as Neumann--Neumann and FETI methods, see~\cite{Pavarino1994,Vasseur2004} and the many references therein. 
We are aware of one work on generalising these methods to $H^2$-norm problems: Brenner and Wang \cite{Brenner2012} considered  iterative substructuring methods for the $h$-version $C^0$ interior penalty discretizations of the biharmonic equation. They show that the usual choices of orthogonalised basis functions required by these algorithms do not extend to the $H^2$-norm context, and that different basis functions must be used on different elements of the mesh.
In comparison, the overlapping and nonoverlapping methods described above generalise straightforwardly to the $H^2$-norm context without additional difficulties. Moreover, a comparison of the computations in \cite{Brenner2005} and \cite{Brenner2012} suggests that the substructuring algorithms only yield a similar performance in practice to the two-level additive Schwarz methods for these problems.

\subsection{Main results}

The numerical scheme of \cite{Smears2014} for fully nonlinear HJB equations leads to a discrete nonlinear problem that can be solved iteratively by a semismooth Newton method. The linear systems obtained from the Newton linearization are generally nonsymmetric but coercive with respect to a discrete $H^2$-type norm.
In section~\ref{sec:lin_hjb}, we apply existing GMRES convergence theory for SPD preconditioners \cite{Eisenstat1983,Loghin2004} to  these nonsymmetric systems, leading to a guaranteed minimum convergence rate, with a contraction factor expressed in terms of the condition number $\kPA$, where $\bo A$ is the matrix of a related symmetric bilinear form that is spectrally equivalent to a discrete $H^2$-type norm, and where $\bo P$ is an arbitrary symmetric positive definite preconditioner.
\cb{Thus, the construction and analysis of preconditioners for a symmetric problem can be used for the solution of the nonsymmetric systems appearing in applications to HJB equations \cite{Smears2014}. A further benefit is that the preconditioner does not require re-assembly at each new semismooth Newton iteration.}

Section~\ref{sec:preconditioners} presents the specific construction of a nonoverlapping additive Schwarz preconditioner $\bo P$ based on $\bo A$, and sections~\ref{sec:approximability} and \ref{sec:stable_decomposition} show the condition number bound
\begin{equation}\label{eq:intro_spectral_bound}
\kPA \lesssim 1+ \frac{p^2 H}{q\, h} +  \frac{p^6 \,H^3}{q^3\, h^3}.
\end{equation}
In comparison to the existing literature, this is the first bound for this class of preconditioners that explicity accounts for the coarse mesh polynomial degree. Unfortunately, \eqref{eq:intro_spectral_bound} implies that this standard class of preconditioners cannot be expected to be robust with respect to the polynomial degree. Nevertheless, our result shows that the coarse space polynomial degree can contribute significantly to reducing the condition number.

The central original result underpinning our analysis is Theorem~\ref{thm:dg_h2_approximation} of section~\ref{sec:approximability}, which shows that for any $v_h\in\Vh$, there is a function $v \in H^2(\Om)\cap H^1_0(\Om)$ such that
\begin{equation}\label{eq:intro_approximation}
\begin{aligned}
\norm{v_h - v}_{L^2(\Om)} + \frac{h}{p} \norm{v_h-v}_{H^1(\Om;\calT_h)} & \lesssim \frac{h^2}{p^2}\absj{v_h}, &
\norm{v}_{H^2(\Om)} & \lesssim \normh{v_h},
\end{aligned}
\end{equation}
where the piecewise Sobolev norms $\norm{\cdot}_{H^s(\Om;\calT_h)}$, jump seminorm $\absj{\cdot}$, and the discrete $H^2$-type norm $\normh{\cdot}$ are defined in section~\ref{sec:definitions}.
This result is a natural converse to classical direct approximation theory, since, here, the nonsmooth function from the discrete space $\Vh$ is approximated by a smoother function from an infinite dimensional space.
It follows from \eqref{eq:intro_approximation} that there exists a function $v_H$ in the coarse space $\VH$, of polynomials of degree $q$ on $\calT_H$, such that
\begin{equation}
\begin{aligned}
\norm{v_h-v_H}_{H^k(\Om;\calT_h)} &\lesssim \frac{H^{2-k}}{q^{2-k}}\normh{v_h}, & k&\in\{0,1,2\},
\end{aligned}
\end{equation}
thus yielding an approximation between coarse and fine meshes that is optimal in the orders of the mesh size and the polynomial degree. The approximation result is used to show the stable decomposition property for the additive Schwarz preconditioner in section~\ref{sec:stable_decomposition}, thereby leading to the spectral bound \eqref{eq:intro_spectral_bound}.

The first numerical experiment, in section~\ref{sec:numexp1}, confirms that \eqref{eq:intro_spectral_bound} is sharp with respect to the orders in the polynomial degrees. The experiment of section~\ref{sec:numexp2} compares nonoverlapping methods with their overlapping counterparts, where it is found that they are competitive in both iteration counts and computational cost.
Despite the polynomial degree suboptimality of these preconditioners, in section~\ref{sec:numexp3} we show computationally that for $h$-refinement, nonoverlapping methods can be efficient and competitive in challenging applications to fully nonlinear HJB equations.


\section{Definitions}\label{sec:definitions}
For real numbers $a$ and $b$, we shall write $a\lesssim b$ to signify that there is a positive constant $C$ such that $a\leq C b$, where $C$ is independent of the quantities of interest, such as the element sizes and polynomial degrees, but possibly dependent on other quantities, such as the mesh regularity parameters.

Let $\Om\subset \R^{\dim}$, $\dim\in\{2,3\}$, be a bounded convex polytopal domain.
Note that convexity of $\Om$ implies that the boundary $\DO$ of $\Om$ is Lipschitz \cite{Grisvard2011}.
Let $\{\calT_h\}_h$ be a sequence of shape-regular meshes on $\Om$, consisting of simplices or parallelepipeds.
For each element $K\in\calT_h$, let $h_K \coloneqq \diam K$. It is assumed that $h = \max_{K\in\calT_h} h_K$ for each mesh $\calT_h$.
Let $\calF^i_h$ denote the set of interior faces of the mesh $\calT_h$ and let $\calF_h^b$ denote the set of boundary faces.
The set of all faces of $\calT_h$ is denoted by $\calFhib \coloneqq \calF_h^i\cup\calF_h^b$. Since each element has piecewise flat boundary, the faces may be chosen to be flat. For $K \in \calT_h$ or $F\in \calFhib$, we use $\pair{\cdot}{\cdot}_K$, respectively $\pair{\cdot}{\cdot}_F$, to denote the $L^2$-inner product over $K$, respectively $F$, of scalar functions, vector fields, and higher-order tensors.
\paragraph{Mesh conditions}
The meshes are allowed to be nonmatching, i.e.\ there may be hanging nodes. We assume that there is a uniform upper bound on the number of faces composing the boundary of any given element; in other words, there is a $c_{\calF}>0$, independent of $h$, such that
\begin{equation}\label{eq:card_F_bound}
\max_{K \in \calT_h} \card \{F \in \calFhib \colon F \subset \p K \} \leq c_{\calF} \qquad\forall K\in\calT_h.
\end{equation}
It is also assumed that any two elements sharing a face have commensurate diameters, i.e.\ there is a $\ch \geq 1$, independent of $h$, such that
\begin{equation}\label{eq:c_h_bound}
 \max ( h_K, h_{K^{\prime}} ) \leq \ch \min( h_K, h_{K^{\prime}} ),
\end{equation}
for any $K$ and $K^{\prime}$ in $\calT_h$ that share a face.
For each $h$, let $\bo p \coloneqq \left( p_K \colon K\in \calT_h\right)$ be a vector of positive integers; note that this requires $p_K \geq 1$ for all $K\in\calT_h$.
We make the assumption that $\bo p$ has \emph{local bounded variation}: there is a $\cp \geq 1$, independent of $h$, such that
\begin{equation}\label{eq:c_p_bound}
 \max( p_K, p_{K^\prime} ) \leq \cp \min( p_K, p_{K^\prime} ),
\end{equation}
for any $K$ and $K^{\prime}$ in $\calT_h$ that share a face.

\paragraph{Function spaces}
For each $K \in \calT_h$, let $\calP_{p_K}(K)$ be the space of all real-valued polynomials in $\R^\dim$ with either total or partial degree at most $p_K$. In particular, we allow the combination of spaces of polynomials of fixed total degree on some parts of the mesh with spaces of polynomials of fixed partial degree on the remainder. We also allow the use of the space of polynomials of total degree at most $p_K$ even when $K$ is a parallelepiped.
The discontinuous Galerkin finite element spaces $\Vh$ are defined by
\begin{equation}\label{eq:cordes_fem_space}
 \Vh \coloneqq \left\{v \in L^2(\Om)\colon \eval{v}{K} \in \calP_{p_K}(K),\;\forall K \in \calT_h \right\}.
\end{equation}
Let $\bo s  \coloneqq \left( s_K \colon K\in \calT_h\right)$ denote a vector of non-negative real numbers. The broken Sobolev space $H^{\bo s}(\Om;\calT_h)$ is defined by
\begin{equation}\label{eq:broken_sobolev}
 H^{\bo s}(\Om;\calT_h) \coloneqq \left\{ v \in L^2(\Om)\colon\eval{v}{K} \in H^{s_K}(K),\;\forall K\in\calT_h \right\}.
\end{equation}
For $s\geq 0$, we set $H^s(\Om;\calT_h) \coloneqq H^{\bo s}(\Om;\calT_h)$, where $s_K = s$ for all $K\in\calT_h$.
The norm $\norm{\cdot}_{H^{\bo s}(\Om;\calT_h)}$ and semi-norm $ \abs{\cdot}_{H^{\bo s}(\Om;\calT_h)}$ are defined on $H^{\bo s}(\Om;\calT_h)$ as
\begin{equation}\label{eq:broken_norms}
\norm{v}_{H^{\bo s}(\Om;\calT_h)}\coloneqq\left(\sum_{K\in\calT_h} \norm{v}_{H^{s_K}(K)}^2\right)^{\frac{1}{2}},
\quad \abs{v}_{H^{\bo s}(\Om;\calT_h)} \coloneqq \left( \sum_{K\in\calT_h} \abs{v}_{H^{s_K}(K)}^2 \right)^{\frac{1}{2}}.
\end{equation}
\cb{For a function $v_h\in\Vh$, the element-wise gradient $\nabla v_h|_K$ and the Hessian $D^2 v_h|_K$ are well-defined for all $K \in \calT_h$ since $v_h$ is smooth on $K$. Thus expressions such as $\pair{D^2 u_h}{D^2 v_h}_K$ are well-defined for all $K\in\calT_h$ and all $u_h$, $v_h \in \Vh$.}

\paragraph{Jump and average operators}
For each face $F\in\calFhib$, let $n_F \in \R^{\dim}$ denote a \emph{fixed} choice of a unit normal vector to $F$. Since $F$ is flat, $n_F$ is constant over $F$.
Let $K$ be an element of $\calT_h$ for which $F\subset \p K$; then $n_F$ is either inward or outward pointing with respect to $K$.
Let $\tau_F\colon H^{s}(K)\tends H^{s-1/2}(F)$, $s>1/2$, denote the trace operator from $K$ to $F$, and let $\tau_F$ be extended componentwise to vector-valued functions.

For each face $F$, define the jump operator $\llb \cdot \rrb$ and the average operator $\lla \cdot \rra$ by
\begin{align*}
  \llb \phi \rrb  &\coloneqq \tau_F \left( \eval{\phi}{K_{\mathrm{ext}}} - \eval{\phi}{K_{\mathrm{int}}} \right),
& \lla \phi \rra &\coloneqq  \frac{1}{2}\tau_F \left( \eval{\phi}{K_{\mathrm{ext}}}  + \eval{\phi}{K_{\mathrm{int}}} \right),
&\text{if }F\in\calF^i_h,
\\
\llb \phi \rrb  &\coloneqq \tau_F \left(\eval{\phi}{K_{\mathrm{ext}}}\right),
&  \lla \phi \rra &\coloneqq  \tau_F \left(\eval{\phi}{K_{\mathrm{ext}}} \right),
&\text{if }F\in\calF_h^b,
\end{align*}
where $\phi$ is a sufficiently regular scalar or vector-valued function, and $K_{\mathrm{ext}}$ and $K_{\mathrm{int}}$ are the elements to which $F$ is a face, i.e.\ $F= \p K_{\mathrm{ext}} \cap \p K_{\mathrm{int}}$.
Here, the labelling is chosen so that $n_F$ is outward pointing with respect to $K_{\mathrm{ext}}$ and inward pointing with respect to $K_{\mathrm{int}}$.
Using this notation, the jump and average of scalar-valued functions, resp.\ vector-valued, are scalar-valued, resp.\ vector-valued.

\paragraph{Tangential differential operators}
For $F \in \calFhib$, let $H^s_{\mathrm{T}}(F)$ denote the space of $H^s$-regular tangential vector fields on $F$, thus $H^s_{\mathrm{T}}(F)\coloneqq\{v\in H^s(F)^\dim\colon v\cdot n_F =0\text{ on }F\}$.
We define the tangential gradient $\nablaT \colon H^s(F) \tends H^{s-1}_{\mathrm{T}}(F)$ and the tangential divergence $\divT \colon H^s_{\mathrm{T}}(F)\tends H^{s-1}(F)$, where $s\geq 1$, following \cite{Grisvard2011}. Let $\{t_i\}_{i=1}^{\dim-1}\subset \R^\dim$ be an orthonormal coordinate system on $F$. Then, for $u \in H^s(F)$ and $v \in H^s_{\mathrm{T}}(F)$ such that $v=\sum_{i=1}^{\dim-1} v_i\,t_i$, with $v_i \in H^s(F)$ for $i=1,\dots,\dim-1$, we define
\begin{equation}\label{eq:tangentialoperators}
\nablaT u  \coloneqq \sum_{i=1}^{\dim-1} t_i \frac{\p u}{\p t_i}, \qquad
\divT v \coloneqq \sum_{i=1}^{\dim-1} \frac{\p v_i}{\p t_i}.
\end{equation}

\paragraph{Mesh-dependent norms}
In the following, we let $u_h$ and $v_h$ denote functions in $\Vh$.
For face-dependent positive real numbers $\mu_F$ and $\eta_F$, let the jump stabilization bilinear form $\Jstab\colon \Vh\times \Vh$ be defined by
\begin{multline}\label{eq:Jstab_scalar}
 \Jstab(u_h,v_h) \coloneqq \sum_{F\in\calF^i_h} \mu_F \pair{\llb \nabla u_h\cdot n_F \rrb}{\llb \nabla v_h\cdot n_F \rrb}_F
\\+ \sum_{F \in\calFhib}\bigl[ \mu_F \pair{\llb \nablaT u_h \rrb}{\llb \nablaT v_h \rrb}_F
+  \eta_F \pair{\llb u_h \rrb}{\llb v_h \rrb}_F\bigr].
\end{multline}
Define the jump seminorm $\absj{\cdot}$ and the mesh-dependent norm $\normh{\cdot}$ on $\Vh$ by
\begin{align}
 \absj{v_h}^2 &\coloneqq  \Jstab(v_h,v_h),&
 \qquad \normh{v_h}^2 &\coloneqq \sum_{K\in\calT_h} \norm{v_h}_{H^2(K)}^2 + \absj{v_h}^2.
\end{align}
For each face $F \in \calFhib$, define
\begin{equation}\label{eq:def_tilde_ph}
\tilde{h}_F \coloneqq
  \begin{cases}
  \min(h_K,h_{K^\prime}), &\text{if }F \in \calF^i_h, \\
  h_K, &\text{if }F \in \calF^b_h,
  \end{cases}
  \qquad
   \tilde{p}_F \coloneqq
  \begin{cases}
  \max(p_K,p_{K^\prime}), &\text{if }F \in \calF^i_h, \\
  p_K, &\text{if }F \in \calF^b_h,
  \end{cases}
\end{equation}
where $K$ and $K^\prime$ are such that $F = \p K \cap \p K^\prime$ if $F\in\calF^i_h$ or $F \subset \p K \cap \DO$ if $F\in\calF^b_h$. The assumptions on the mesh and the polynomial degrees, in particular \eqref{eq:c_h_bound} and \eqref{eq:c_p_bound}, show that if $F$ is a face of an element $K$, then $ h_K \leq \ch \,\tilde{h}_F$ and $\tilde{p}_F\leq \cp\, p_K$.
Henceforth, it is assumed that the parameters $\mu_F$ and $\eta_F$ in \eqref{eq:Jstab_scalar} are given by 
\begin{equation}\label{eq:eta_mu}
\begin{aligned}
 \mu_F & \coloneqq \cmu \frac{\tilde{p}_F^2}{\tilde{h}_F}, &  \eta_F &\coloneqq \ceta \frac{\tilde{p}_F^6}{\tilde{h}^3_F}& &\forall\,F\in\calFhib,
 \end{aligned}
\end{equation}
where $\cmu$ and $\ceta$ are fixed positive constants independent of $h$ and $\bo p$.

\paragraph{Approximation}%
Under the hypothesis of shape-regularity of $\{\calT_h\}$, for any function $u\in H^{\bo s}(\Om;\calT_h)$, there exists an approximation $\Pi_h u \in \Vh$, such that for each element $K\in \calT_h$,
\begin{subequations}\label{eq:approximation_bound}
\begin{equation}\label{eq:approximation_element}
 \norm{u-\Pi_h u}_{H^r(K)} \lesssim \frac{h_K^{\min(s_K,\,p_K+1)-r}}{p_K^{s_K-r}}\,\norm{u}_{H^{s_K}(K)} \quad \forall\,r, \, 0\leq r \leq s_K,
\end{equation}
and, if $s_K >1/2$,
\begin{equation}\label{eq:approximation_face}
 \norm{D^{\a}\left(u-\Pi_h u\right)}_{L^2(\p K)} \lesssim \frac{h_K^{\min(s_K,\,p_K+1)-\abs{\a}-1/2}}{p_K^{s_K-\abs{\a}-1/2}}\,\norm{u}_{H^{s_K}(K)} 
\quad \forall \a,\,\abs{\a} \leq k,
\end{equation}
\end{subequations}
where $k$ is the greatest non-negative integer strictly less than $s_K-1/2$. The constants in \eqref{eq:approximation_element} and \eqref{eq:approximation_face} do not depend on $u$, $K$, $p_K$, $h_K$ or $r$, but depend possibly on $\max_{K\in\calT_h} s_K$. Vector fields can be approximated componentwise.

\section{HJB equations}\label{sec:lin_hjb}
We consider fully nonlinear HJB equations of the form
\begin{equation}\label{eq:HJB_PDE}
\begin{aligned}
\sup_{\a\in\Ld}\left[ L^\a u - f^\a \right] &=0 & &\text{in }\Om,\\
u &= 0  & &\text{on }\DO,
\end{aligned}
\end{equation}
where $\Om$ is a bounded convex domain, $\Ld$ is a compact metric space, and the operators $L^\a$ are given by
\begin{equation}\label{eq:L_def}
\begin{aligned}
L^\a v &\coloneqq  a^{\a}\colon D^2 v  +  b^\a \cdot \nabla v - c^\a\,v, &  v & \in H^2(\Om),\;\a\in\Ld.
\end{aligned}
\end{equation}
For simplicity of presentation here, we restrict our attention to the case $b^\a\equiv 0$, $c^\a \equiv 0$, and refer the reader to \cite{Smears2014} for the general case. The matrix-valued function $a$ and the scalar function $f$ are assumed to be continuous on $\Ob\times\Ld$, and $a$ is assumed to be uniformly elliptic, uniformly over $\Ob\times \Ld$.
The PDE in \eqref{eq:HJB_PDE} is fully nonlinear in the sense that the Hessian of the unknown solution appears inside the nonlinear term in \eqref{eq:HJB_PDE}. As a result of the nonlinearity, no weak form of the equation is available: this has constituted a long-standing difficulty in the development of high-order methods for this class of problems. 

However, provided that the coefficients of $L^\a$ satisfy the Cordes condition, which, in the case of pure diffusion, requires that there exist $\eps\in(0,1]$ such that
\begin{equation}\label{eq:cordes}
\frac{\abs{a^\a(x)}^2}{\left(\Trace a^\a(x)\right)^2} \leq \frac{1}{\dim - 1 + \eps} \quad\forall\,\a\in\Ld,\; \forall\,x\in\Om.
\end{equation}
then the boundary-value problem \eqref{eq:HJB_PDE} has a unique solution in $H^2(\Om)\cap H^1_0(\Om)$, see \cite[Theorem~3]{Smears2014}. Observe that for problems in two spatial dimensions, condition~\eqref{eq:cordes} is equivalent to uniform ellipticity.

Defining the operator $\Fg[u]\coloneqq \sup_{\a\in\Ld}\left[ \gamma^\a (L^\a u-f^\a)\right]$, where $\gamma^\a = \Trace a^\a / \abs{a^\a}^2$, the numerical scheme of \cite{Smears2014} for solving \eqref{eq:HJB_PDE} associated to a homogeneous Dirichlet boundary condition is to find $u_h\in\Vh$ such that
\begin{equation}\label{eq:HJB_scheme}
\begin{aligned}
\calA_h(u_h;v_h) & = 0 & \forall\,v_h\in\Vh,
\end{aligned}
\end{equation}
where the nonlinear form $\calA_h$ is defined in \cite[Eq.~(5.3)]{Smears2014}, and can be equivalently given as
\begin{multline}
\calA_h(u_h;v_h)\coloneqq \sum_{K\in\calT_h} \pair{\Fg[u_h]}{\Delta v_h}_K \\ + \frac{1}{2}\left( a_h(u_h,v_h) - \sum_{K\in\calT_h} \pair{\Delta u_h}{\Delta v_h}_K + \Jstab(u_h,v_h)\right),
\end{multline}
where the bilinear form $a_h\colon\Vh\times\Vh\tends \R$ is defined by
\begin{multline}\label{eq:a_h_definition}
a_h(u_h,v_h) \coloneqq \sum_{K\in\calT_h} \pair{D^2u_h}{D^2 v_h}_{K}  + \Jstab(u_h,v_h)
\\ + \sum_{F\in \calFhi} \left[ \pair{\divT \nablaT \lla u_h \rra}{\llb \nabla v_h\cdot n_F \rrb}_F + \pair{\divT \nablaT \lla v_h \rra}{\llb \nabla u_h\cdot n_F \rrb}_F\right]
\\ - \sum_{F\in\calFhib} \left[\pair{\nablaT \lla \nabla u_h\cdot n_F\rra}{\llb \nablaT v \rrb}_F + \pair{\nablaT \lla \nabla v_h\cdot n_F\rra}{\llb \nablaT v_h \rrb}_F\right].
\end{multline}

\subsection{\cb{Semismooth Newton method}}
\cb{In~\cite[Section~8]{Smears2014}, it is shown that the discretized nonlinear problem \eqref{eq:HJB_scheme} can be solved by a semismooth Newton method, which leads to a sequence of nonsymmetric but positive definite linear systems to be solved at each iteration.
We summarize here the essential ideas on the semismooth Newton, and refer the reader to~\cite{Smears2014} for the complete analysis.

For $x\in \Omega$ and $M \in \R^{\dim\times\dim}_{\mathrm{sym}}$, define $\Fg(x,M)\coloneqq \sup_{\a\in\Ld}[\gamma^{\a}(x)\left(a^{\a}(x){:}M - f^{\a}(x)\right)]$, and let $\Ld(x,M)$ denote the set of all $\a \in \Ld$ that attain the supremum in $\Fg(x,M)$; note that $\Ld(x,M)$ is always a non-empty subset of $\Ld$ due to the compactness of $\Ld$ and the continuity of the functions $a$, $f$ and $\gamma$ over $\Ob\times \Ld$. This defines a set-valued mapping $(x,M)\mapsto \Ld(x,M)$.
For a function $v\in H^2(\Om;\calT_h)$, let $\Ld[v]$ denote the set of all Lebesgue measurable mappings $\alpha(\cdot)\colon \Om\tends \Ld$ that satisfy $\alpha(x) \in \Ld(x,D^2 v(x))$ for almost every $x\in \Om$; in~\cite[Theorem~10]{Smears2014}, it is shown that $\Ld[v]$ is non-empty for any $v\in H^2(\Om;\calT_h)$.

The semismooth Newton method is now defined as follows. Start by choosing an initial iterate $u_h^0 \in \Vh$. Then, for each nonnegative integer $j$, given the previous iterate~$u_h^j\in\Vh$, choose an $\a_j \in \Ld[u_h^j]$.
Next, the function $f^{\a_j} \colon \Om \tends \R$ is defined by $f^{\a_j}\colon x\mapsto f^{\a_j(x)}(x)$; the functions $a^{\a_j}$ and $\gamma^{\a_j}$ are defined in a similar way. Note that the measurability of the mappings $\a_j$ ensures the measurability of $f^{\a_j}$, $a^{\a_j}$ and $\gamma^{\a_j}$.
Then, find the solution $u^{j+1}_h\in\Vh$ of the linearized system
\begin{equation}\label{eq:semi_newton_scheme}
 B_h^j(u_h^{j+1},v_h)=\sum_{K\in\calT_h}\pair{\gamma^{\a_j}f^{\a_j}}{\Delta v_h}_K \quad\forall\,v_h\in\Vh,
\end{equation}
where the bilinear form $B_h^j\colon\Vh\times\Vh\tends \R$ is defined by
\begin{multline}\label{eq:bhj_def}
B_h^j(w_h,v_h) \coloneqq \sum_{K\in\calT_h} \pair{\gamma^{\a_j} a^{\a_j}{:} D^2 w_h}{\Delta v_h}_K   \\ + \frac{1}{2}\left( a_h(u_h,v_h) - \sum_{K\in\calT_h} \pair{\Delta u_h}{\Delta v_h}_K + \Jstab(u_h,v_h)\right),
 \end{multline}
In \cite[Theorem~11]{Smears2014}, it was shown that $u_h^j\tends u$ as $j\tends \infty$ for a sufficiently close initial guess $u_h^0$, and moreover that the convergence is superlinear. It was also shown that the bilinear forms $B_h^j$ are uniformly bounded and coercive in an $H^2$-type norm, with constants independent of the iterates. Since the preconditioners of this work take advantage of the coercivity of the $B_h^j$, we summarize the relevant results in the following lemma.}
\begin{lemma}\label{lem:a_h_coercivity}
Let $\Om$ be a bounded convex polytopal domain and let $\{\calT_h\}_h$ be a shape-regular sequence of meshes satisfying \eqref{eq:card_F_bound}. Let the bilinear forms $B_h^j$ be defined by \eqref{eq:bhj_def}. Then, there exist positive constants $\underline{c}_{\mu}$ and $\underline{c}_{\eta}$ such that if $\cmu \geq \underline{c}_{\mu}$ and $\ceta\geq\underline{c}_{\eta}$, then the bilinear forms $a_h$ and $B_h^j$ are uniformly coercive: for all $v_h$, $w_h\in\Vh$, we have
\begin{align}
\normh{v_h}^2 &\lesssim a_h(v_h,v_h), &  & & \abs{a_h(v_h,w_h)}\lesssim \normh{v_h}\norm{w_h}_{h,2},\label{eq:poincare} \\
\normh{v_h}^2 &\lesssim B_h^j(v_h,v_h), & & & \abs{B_h^j(v_h,w_h)}\lesssim \normh{v_h}\normh{w_h},\label{eq:bhj_coercivity}
\end{align}
where the constants are independent of the sequence $\{u_h^j\}_{j=0}^\infty$ and of the choice of the mappings $\a_j \in \Ld[u_h^j]$ for each $j\geq 0$.
\end{lemma}
\begin{proof}
First we prove \eqref{eq:poincare}. The continuity bound in \eqref{eq:poincare} is a straightforward consequence of the trace and inverse inequalities. To show the coercivity bound, we first show that 
\begin{equation}\label{eq:poincare_2}
\normh{v_h}^2 \lesssim  \sum_{K\in\calT_h} \norm{D^2 v_h}_{L^2(K)}^2 + \absj{v_h}^2  \eqqcolon \abs{v_h}_{h,2}^2.
\end{equation}
For any $v_h\in\Vh$, integration by parts gives
\begin{multline}
\sum_{K\in\calT_h} \norm{\nabla v_h}_{L^2(K)}^2 = \sum_{K\in\calT_h}\pair{v_h}{-\Delta v_h}_K + \sum_{F\in\calFhib}\pair{\llb v_h \rrb}{\lla \nabla v_h\cdot n_F\rra}_F\\ + \sum_{F\in\calFhi}\pair{\lla  v_h \rra}{\llb \nabla v_h\cdot n_F\rrb}_F.
\end{multline}
Hence, the trace and inverse inequalities imply that
\begin{equation}\label{eq:poincare_3}
\sum_{K\in\calT_h} \norm{\nabla v_h}_{L^2(K)}^2  \lesssim \norm{v_h}_{L^2(\Om)} \abs{v_h}_{h,2}.
\end{equation}
We recall the broken Poincar\'e inequality
\begin{equation}
\norm{v_h}_{L^2(\Om)}^2 \lesssim \sum_{K\in\calT_h}\norm{\nabla v_h}_{L^2(K)}^2 + \sum_{F\in\calFhib}\frac{1}{\tilde{h}_F} \norm{\llb v_h \rrb}_{L^2(F)}^2.
\end{equation}
Therefore it follows from \eqref{eq:poincare_3} that we have $\sum_{K\in\calT_h} \norm{v_h}_{H^1(K)}^2 \lesssim \abs{v_h}_{h,2}^2$, from which we deduce \eqref{eq:poincare_2}.
The proof of \eqref{eq:poincare} is now completed by noting that \cite[Lemma~7]{Smears2013} implies that there exist $\underline{c}_{\mu}$ and $\underline{c}_{\eta}$ such that $\abs{v_h}_{h,2}^2 \lesssim a_h(v_h,v_h)$, whenever $\cmu \geq \underline{c}_{\mu}$ and $\ceta\geq\underline{c}_{\eta}$, \cb{since $a_h$ equals the bilinear form denoted by $B_{\rm{DG}(1)}$ in the notation of \cite[Lemma~7]{Smears2013}}.
\cb{The continuity bound for $B_h^j$ in~\eqref{eq:bhj_coercivity} can also be shown straightforwardly through the Cauchy--Schwarz inequality with the trace and inverse inequalities, where we note that the functions $\gamma^{\a_j}$, respectively $a^{\a_j}$, appearing in \eqref{eq:bhj_def} are uniformly bounded in $L^\infty$ by $\norm{\gamma}_{C(\Ob\times \Ld)}$, respectively by $\norm{a}_{C(\Ob\times\Ld;\R^{\dim\times\dim})}$; this implies that the continuity constants in \eqref{eq:bhj_coercivity} can be taken to be independent of the $\{u_h^j\}_{j=0}^\infty$. The coercivity bound~in \eqref{eq:bhj_coercivity} was shown in~\cite[Theorem~8]{Smears2013} and \cite[Eq.~(8.5)]{Smears2014}, where it is seen that the coercivity constant is independent of the iteration count $j$, but otherwise may depend on the constant $\eps$ from \eqref{eq:cordes} and the choice of the penalty parameters $\cmu$ and $\ceta$.}
\end{proof}

\subsection{Iterative solution by the preconditioned GMRES method}
Each step of the semismooth Newton method requires the solution of~\eqref{eq:semi_newton_scheme}.
These linear systems have a common general form, which consists of finding $\tilde{u}_h \in \Vh$ such that
\begin{equation}\label{eq:linearized_nonsymm}
\begin{aligned}
B_h(\tilde{u}_h,v_h) = \ell_h(v_h) & & & \forall \,v_h\in\Vh,
\end{aligned}
\end{equation}
where we shall henceforth omit to denote the dependence of the bilinear form $B_h$ and of the right-hand side $\ell_h$ on the iteration number of the semismooth Newton method.
It follows from Lemma~\ref{lem:a_h_coercivity} that there exist positive constants $c_{B}$ and $C_B$ such that, for any $v_h$ and $w_h\in \Vh$,
\begin{equation}\label{eq:a_b_equivalence}
\begin{aligned}
 a_h(v_h,v_h) \leq \frac{1}{c_B} B_h(v_h,v_h),& & & \abs{B_h(v_h,w_h)}\leq C_B\sqrt{a_h(v_h,v_h)}\sqrt{a_h(w_h,w_h)},
\end{aligned}
\end{equation}
where $c_B$ and $C_B$ are independent of the iteration count of the semismooth Newton method and the discretization parameters. Therefore the sequence of linearisations of \eqref{eq:HJB_scheme} are uniformly bounded and coercive with respect to the norm defined by the bilinear form $a_h$.

The coercivity and boundedness of $B_h$ imply that an efficient preconditioner for $a_h$ can also be used effectively as a preconditioner for the GMRES algorithm applied to \eqref{eq:linearized_nonsymm}.
 Indeed, assume that $\bo P$ is an SPD preconditioner for the matrix $\bo A\coloneqq\left(a_h(\phi_i,\phi_j)\right)$ that satisfies
\begin{equation}\label{eq:hypo_p_a_equivalence}
\begin{aligned}
0<c_{\bo P}\leq  \frac{ \bo v^\top \bo A \bo v}{\bo v^\top \bo P \bo v} \leq C_{\bo P} & & & \forall\, \bo v \in \R^{\Dim \Vh}\setminus\{0\},
\end{aligned}
\end{equation}
where we assume that $c_{\bo P}$ and $C_{\bo P}$ are the best possible constants in \eqref{eq:hypo_p_a_equivalence}.
Thus the condition number $\kPA = C_{\bo P} / c_{\bo P}$.
Let the matrix $\bo B \coloneqq \left( B_h(\phi_j,\phi_i)\right)$.
Then, the preconditioner $\bo P$ can be used in either the right or left preconditioned GMRES method \cite{Saad2003,Saad1986} for solving \eqref{eq:linearized_nonsymm} as follows. First, we define the norms $\norm{\cdot}_{\bo P}$ and $\norm{\cdot}_{\bo P^{-1}}$ on $\R^{\Dim \Vh}$ by 
\begin{equation}
\begin{aligned}
\norm{\bo v}^2_{\bo P} \coloneqq \bo v^\top \bo P\, \bo v, & & & \norm{\bo v}^2_{\bo P^{-1}} \coloneqq \bo v^\top \bo P^{-1}\, \bo v & & &\forall\,\bo v\in \R^{\Dim \Vh}.\end{aligned}
\end{equation}
Applying $k$-steps of the right preconditioned GMRES method in the $\bo P^{-1}$-inner product computes $\bo u_k$ as the solution of 
\begin{equation}\label{eq:GMRES_right_prec}
\begin{aligned}
\bo u_k = \bo P^{-1}\bo  w_k, & & &
\bo w_k=\argmin_{\bo{\tilde{w}}_k \in \calK_k(\bo B \bo P^{-1},\bo r_0)+\bo w_0} \norm{ \bo B \bo P^{-1} \left( \bo w - \bo{\tilde{w}}_k\right)}_{\bo P^{-1}},
\end{aligned}
\end{equation}
where $\bo r_0$ denotes the initial residual, $\bo w \coloneqq \bo P \bo u$, $\bo w_0\coloneqq \bo P \bo u_0$, and where $\calK(\bo B \bo P^{-1},\bo r_0)$ denotes the $k$-dimensional Krylov subspace generated by $\bo B \bo P^{-1}$ and $\bo r_0$. 
It is well-known that \eqref{eq:GMRES_right_prec} is equivalent to
\begin{equation}\label{eq:GMRES_left_prec}
\bo u_k = \argmin_{\bo{\tilde{u}}_k \in 	\calK_k(\bo P^{-1}\bo B,\bo P^{-1} \bo r_0)+\bo u_0}\norm{ \bo P^{-1}\bo B \left(\bo u - \bo{\tilde{u}}_k\right)}_{\bo P},
\end{equation}
which is obtained after $k$-steps of the left-preconditioned GMRES algorithm in the $\bo P$-inner product, see~\cite{Saad2003}. For a discussion of the implementation of the preconditioned GMRES method in the $\bo P$- and $\bo P^{-1}$-inner products, we refer the reader to \cite[p.~269]{Saad2003}.

It follows from \eqref{eq:a_b_equivalence} and the hypothesis \eqref{eq:hypo_p_a_equivalence} that $\bo B$ is also coercive and bounded in the norm defined by $\bo P$: for any $\bo v$ and $\bo w\in \R^{\Dim\Vh}$, we have
\[
\begin{aligned}
\norm{\bo v}_{\bo P}^2 \leq \frac{1}{c_{\bo P}c_B} \bo v^\top \bo B \bo v, & & & \abs{ \bo v^\top \bo B \bo w } \leq  C_{\bo P} C_B \norm{\bo v}_{\bo P} \norm{\bo w}_{\bo P}.
\end{aligned}
\]
This enables us to appeal to the following well-known bound from GMRES convergence theory~\cite{Eisenstat1983}.

\begin{theorem}\label{thm:GMRES_bound}
Let $\bo u\in \R^{\Dim\Vh}$ be the vector representing the solution of \eqref{eq:linearized_nonsymm}. For each $k\geq 1$, let $\bo u_k$ be defined by \eqref{eq:GMRES_right_prec} or equivalently by \eqref{eq:GMRES_left_prec}, with associated residual $\bo r_k$. Then
\begin{equation}\label{eq:GMRES_bound}
\frac{\norm{\bo r_k}_{\bo P^{-1}}}{\norm{\bo r_0}_{\bo P^{-1}}} =  \frac{\norm{ \bo P^{-1} \bo r_k}_{\bo P}}{\norm{\bo P^{-1} \bo r_0}_{\bo P}}
 \leq \left(1  - \frac{c_{\bo P}^2 c_B^2 }{ C_{\bo P}^2 C_B^2} \right)^{k/2} = \left(1- \frac{1}{\kPA^2}\frac{c_B^2}{C_B^2}\right)^{k/2}.
\end{equation}
\end{theorem}
The bound~\eqref{eq:GMRES_bound} and the coercivity of $\bo B$ imply  the following bound for the error:
\[
\norm{\bo u - \bo u_k}_{\bo P}^2 \leq \frac{1}{c_{\bo P} c_{B}} (\bo u - \bo u_k)^\top \bo r_k \leq \frac{1}{c_{\bo P}c_{B}} \norm{\bo u - \bo u_k}_{\bo P}\norm{\bo r_k}_{\bo P^{-1}},
\]
thereby implying that
\begin{equation}\label{eq:GMRES_converge}
\norm{\bo u - \bo u_k}_{\bo P} \leq \frac{\norm{\bo r_0}_{\bo P^{-1}}}{c_{\bo P}c_{B}} \left(1-\frac{c_B^2}{C_B^2\, \kPA^2}\right)^{k/2}.
\end{equation}
The bound \eqref{eq:GMRES_converge} gives a guaranteed minimum convergence rate for GMRES in the $\bo P$-norm, which is equivalent to the $a_h$-norm up to the condition number $\kPA$. 
We recall that $a_h$ defines a norm equivalent to $\normh{\cdot}$, which is the norm of interest, as shown by Lemma~\ref{lem:a_h_coercivity}.
This strongly suggests that the $\bo P^{-1}$-norm, as opposed to the Euclidean norm, of the residual is a natural objective to be minimized by GMRES, as in \eqref{eq:GMRES_right_prec} and \eqref{eq:GMRES_left_prec}.

The conclusion from \eqref{eq:GMRES_bound} and \eqref{eq:GMRES_converge} is that, if $\bo P$ is a robust preconditioner for $\bo A$ in the sense of yielding uniformly bounded condition numbers with respect to the parameters being varied, then $\bo P$ will also be a robust preconditioner for the nonsymmetric problems arising from linearizations of HJB equations. 
In section~\ref{sec:preconditioners}, we construct a specific symmetric positive definite preconditioner~$\bo P$, based on a nonoverlapping domain decomposition method, that will be used to solve \eqref{eq:semi_newton_scheme}.

\begin{remark}\label{rem:GMRES_descriptive}
The general preconditioning strategy proposed here was largely motivated by the analysis in {\upshape\cite{Loghin2004}}.
It is well-known {\upshape\cite{Wathen2015}} that convergence bounds for GMRES, such as \eqref{eq:GMRES_bound}, need not be descriptive of the convergence rate obtained in practice, i.e.\ GMRES may perform significantly better than what is predicted by \eqref{eq:GMRES_bound} alone. In particular this is observed in some of the experiments of section~{\upshape\ref{sec:numexp3}} below.
This implies that the efficiency of the preconditioners must generally be assessed from computations.
\end{remark}

\subsection{Condition number of the unpreconditioned problem}
The condition number of the matrix $\bo A\coloneqq\left(a_h(\phi_i,\phi_j)\right)$ depends on the choice of basis for $\Vh$. However, in practice, the basis is often chosen to be either a nodal basis or a mapped orthonormal basis. For example, let us assume that each basis function $\phi_i$ of $\Vh$ has support in only one element, and is mapped from a member of a set of functions that are $L^2$-orthonormal on a reference element. Then, arguments that are similar to those in \cite{Antonietti2011} show that the $\ell^2$-norm condition number $\kappa(\bo A)$ of the matrix $\bo A\coloneqq\left(a_h(\phi_i,\phi_j)\right)$ satisfies
\begin{equation}\label{eq:unpreconditioned_bound}
\kappa\left(\bo A\right) \lesssim \max_{K\in\calT_h} \frac{p_K^8}{h_K^4} \,\frac{\max_{K\in\calT_h} h_K^{\dim}}{\min_{K\in\calT_h} h_K^{\dim}},
\end{equation}
where it is recalled that $\dim$ is the dimension of the domain $\Om$.

\section{Domain decomposition preconditioners}\label{sec:preconditioners}
Let $\Om$ be partitioned into a set $\calT_{S}\coloneqq\{\Om_i\}_{i=1}^N$ of nonoverlapping Lipschitz polytopal subdomains $\Om_i$. The partition $\calT_S$ is assumed to be conforming. A coarse simplicial or parallelepipedal mesh $\calT_H$ is associated to each fine mesh $\calT_h$. Let $H_D\coloneqq\diam D$ for each $D\in \calT_H$ and suppose that $H\coloneqq\max_{D\in\calT_H} H_D$. It is required that the sequence of meshes $\{\calT_H\}_H$ satisfy the mesh conditions of section~\ref{sec:definitions}. Furthermore, the partitions $\calT_S$, $\calT_H$ and $\calT_h$ are assumed to be \emph{nested}, in the sense that no face of $\calT_S$, respectively $\calT_H$, cuts the interior of an element of $\calT_H$, respectively $\calT_h$. Hence, each element $D\in\calT_H$ satisfies $\overline{D} = \bigcup \overline{K}$, where the union is over all elements $K\in\calT_h$ such that $K\subset D$.

For each mesh $\calT_H$, let $\bo q\coloneqq\left( q_D\colon D\in\calT_H\right)$ be a vector of \emph{positive} integers; so $q_D\geq 1$ for each element $D\in\calT_H$. Assume that $\bo q$ satisfies the bounded variation property of \eqref{eq:c_p_bound}, and that $q_D \leq \min_{K\subset D} p_K$ for all $D\in\calT_H$.
For each $D\in \calT_H$, define the sets
\begin{equation}
\begin{aligned}
 \calT_h(D) &\coloneqq\left\{K \in \calT_h\colon K\subset D\right\}, &
 \calF_h^i(D)&\coloneqq\left\{F\in\calF_h^i\colon F\subset D\right\},
\\ \calF_h^i(\p D)&\coloneqq\left\{F\in\calF_h^i\colon F\subset \p D\right\}, &
 \calFhib(\p D) &\coloneqq \{F\in\calFhib \colon F\subset \p D \}.
\end{aligned}
\end{equation}
Although the sets $\calF_h^i(D)$ and $\calFhib(D)$ are not disjoint, the above assumptions on the meshes imply that $\calFhib = \bigcup_D \calF_h^i(D)\cup \calFhib(\p D)$ and that $\calF_h^i = \bigcup_D \calF_h^i(D)\cup \calF_h^i(\p D)$.
Define the function spaces
\begin{subequations}
\begin{align}
\Vh^i &\coloneqq \left\{ v \in L^2(\Om_i)\colon \eval{v}{K}\in \calP_{p_K}(K)\quad\forall\,K\in\calT_h,K\subset \Om_i\right\},\quad 1\leq i \leq N, \\ 
\VH &\coloneqq \left\{ v\in L^2(\Om)\colon \eval{v}{D}\in \calP_{q_D}(D)\quad\forall\, D\in\calT_H\right\}.
\end{align}
\end{subequations}
For convenience of notation, let $\Vh^0\coloneqq\VH$.
It follows from the above conditions on the meshes that every function $v_H \in \VH$ also belongs to $\Vh$, so let $I_0 \colon \VH \tends \Vh$ denote the natural imbedding map.
For $1\leq i \leq N$, let $I_i\colon \Vh^i\tends \Vh$ denote the natural injection operator defined by
\begin{equation}\label{eq:injection_operator_i}
I_i \,v_i \coloneqq \begin{cases}
 v_i &\text{on }\Om_i,
 \\ 0 &\text{on }\Om-\Om_i,
 \end{cases}
 \qquad\forall\, v_i \in \Vh^i.
\end{equation}
Then, any function $v_h \in \Vh$ can be decomposed as
$
v_h = \sum_{i=1}^N I_i \left(\eval{v_h}{\Om_i}\right).
$
Let the bilinear forms $a_h^i\colon\Vh^i\times\Vh^i\tends \R$, $0\leq i \leq N$, be defined by
\begin{align}
a_h^i(u_i,v_i)&\coloneqq a_h(I_i\,u_i,I_i\,v_i)\qquad \forall \, u_i, v_i\in \Vh^i.
\end{align}
It is clear that the bilinear forms $a_h^i$ are symmetric and coercive on $\Vh^i\times \Vh^i$.
\cb{For each $0 \leq i \leq N$, let $\bo A_i$ denote the matrix that corresponds to the bilinear form $a_h^i$ and let $\bo I_i$ denotes the matrix corresponding to the injection operator $I_i$. Therefore, for each $0\leq i \leq N$, the matrix $\bo A_i$ has dimension $\Dim \Vh^i \times \Dim \Vh^i$, and the matrix $\bo I_i$ has dimension $\Dim \Vh \times \Dim \Vh^i$.
Then, we define $\bo P_{i}^{-1} \coloneqq \bo{I}_i\, \bo A_i^{-1} \,\bo{I}_i^{\top}$, which therefore has dimension $\Dim \Vh \times \Dim \Vh$.}

The additive Schwarz preconditioner $\bo P$ is defined in terms of its inverse by
\begin{equation}
	\bo P^{-1} \coloneqq\sum_{i=0}^N \bo P_i^{-1}.
\end{equation}
Thus $\bo P^{-1}$ defines a symmetric positive definite preconditioner $\bo P$ that may be used as explained in section~\ref{sec:lin_hjb}.
Further preconditioners, such as multiplicative, symmetric multiplicative and hybrid methods, are presented in \cite{Smith1996,Toselli2005} and the references therein.
The general theory of Schwarz methods \cite{Smith1996,Toselli2005} simplifies the analysis of these preconditioners to the verification of three key properties.
\begin{property}\label{prop:stable_decomposition}
Suppose that there exists a positive constant $c_0$ such that each $v_h\in\Vh$ admits a decomposition $v_h = \sum_{i=0}^N I_i\, v_i$, with $v_i\in\Vh^i$, for each $0\leq i \leq N$, with
\begin{equation}\label{eq:stable_decomposition}
\sum_{i=0}^N a_h^i(v_i,v_i) \leq c_0\, a_h(v_h,v_h).
\end{equation}
\end{property}

\begin{property}\label{prop:strengthened_cauchy}
Assume that there exist constants $\eps_{ij}\in\left[0,1\right]$, such that
\begin{equation}\label{eq:strengthened_cauchy}
\abs{a_h(I_i\, v_i, I_j v_j)}\leq \eps_{ij} \sqrt{a_h(I_i \,v_i,I_i\, v_i)\; a_h(I_j \,v_j,I_j\, v_j)}, 
\end{equation}
for all $v_i\in\Vh^i$ and all $v_j\in\Vh^j$,  $1\leq i,j\leq N$. Let $\rho(\mathcal{E})$ denote the spectral radius of the matrix $\mathcal{E}\coloneqq(\eps_{ij})$.
\end{property}

\begin{property}\label{prop:local_stability}
Suppose that there exists a constant $\omega \in (0,2)$, such that
\begin{equation}\label{eq:equivalence_local_operators}
a_h(I_i\, v_i, I_i\, v_i) \leq \omega \,a_h^i(v_i,v_i)\quad\forall\, v_i\in\Vh^i,\; 0\leq i \leq N.
\end{equation}
\end{property}

Properties~\ref{prop:stable_decomposition}--\ref{prop:local_stability} are sometimes referred to respectively as the stable decomposition property, the strengthened Cauchy--Schwarz inequality, and local stability.

The following theorem from the theory of Schwarz methods is quoted from \cite{Toselli2005}.
\begin{theorem}\label{thm:schwarz_condition}
If Properties~\ref{prop:stable_decomposition}--\ref{prop:local_stability} hold, then the condition number $\kPA$ obtained by the additive Schwarz preconditioner satisfies
\begin{equation}
\kPA \leq c_0 \,\omega \left(\rho(\mathcal{E})+1\right).
\end{equation}
\end{theorem}
\begin{remark}\label{rem:schwarz_constants}
With the above choices of bilinear forms $a_h^i$ and with the arguments presented in {\normalfont \cite{Antonietti2011}}, it is seen that \eqref{eq:equivalence_local_operators} holds in fact with equality for $\omega =1$. Also, in \eqref{eq:strengthened_cauchy}, we can take $\eps_{ij}=1$ if $\p\Om_i \cap \p \Om_j \neq \emptyset$, and $\eps_{ij}=0$ otherwise. Therefore, as explained in {\normalfont\cite{Antonietti2011}}, $\rho\left(\mathcal{E}\right) \leq N_c+1$, where $N_c$ is the maximum number of adjacent subdomains that a given subdomain might have. Therefore, Properties~{\upshape\ref{prop:strengthened_cauchy}}~and~{\upshape\ref{prop:local_stability}} hold, and it remains to verify Property~\upshape{\ref{prop:stable_decomposition}}.
\end{remark}

The following theorem determines a bound on the constant appearing in \eqref{eq:stable_decomposition}, which can be used in conjunction with Theorem~\ref{thm:schwarz_condition} to analyse the properties of the preconditioners. The proof of this result is given in the following sections. 

\begin{theorem}\label{thm:stable_decomposition}
 Let $\Om\subset \R^{\dim}$, $\dim\in\{2,3\}$, be a bounded convex polytopal domain, and let $\calT_S$, $\{\calT_H\}_H$ and $\{\calT_h\}_h$ be successively nested shape-regular sequences of meshes, with $\calT_S$ conforming, and $\{\calT_H\}_H$ and $\{\calT_h\}_h$ satisfying \eqref{eq:card_F_bound}, \eqref{eq:c_h_bound} and \eqref{eq:c_p_bound}. Let $\mu_F$ and $\eta_F$ satisfy \eqref{eq:eta_mu} for each face $F$, with $\cmu$ and $\ceta$ chosen to satisfy the hypothesis of Lemma~{\upshape\ref{lem:a_h_coercivity}}. Then, each $v_h\in\Vh$ admits a decomposition $v_h = \sum_{i=0}^N I_i\, v_i$, with $v_i\in\Vh^i$, $0\leq i \leq N$, such that
\begin{equation}
\sum_{i=0}^N a_h^i(v_i,v_i) \lesssim \tilde{c}_0 \, a_h(v_h,v_h), \\
\end{equation}
where the constant $\tilde{c}_0$ is given by
\begin{multline}\label{eq:c_0_constant}
 \tilde{c}_0  \coloneqq1 + \max_{D\in\calT_H} \left[\frac{q_D}{H_D}\max_{K \in \calT_h(D)} \frac{p_K^2}{h_K}\right]\max_{D\in\calT_H}\frac{H_D^2}{q_D^2}
 \\ + \max_{D\in\calT_H}\left[ \frac{q_D}{H_D} \max_{K\in\calT_h(D)} \frac{p_K^6}{h_K^3}\right]\max_{D\in\calT_H}\frac{H^4_D}{q_D^4}.
\end{multline}
\end{theorem}

It follows from Theorems~\ref{thm:schwarz_condition}~and~\ref{thm:stable_decomposition} that the condition number satisfies
\begin{equation}\label{eq:condition_preconditioned}
\kPA \lesssim \tilde{c}_0\left(N_c+2\right) ,
\end{equation}
where $\tilde{c}_0$ is given in \eqref{eq:c_0_constant} above, and $N_c$ is the maximum number of adjacent subdomains that a given subdomain from $\calT_S$ might have.
\cb{Thus the condition number does not depend on the number $N$ of subdomains, but may depend on the maximum number of neighbours any subdomain possesses, denoted by $N_c$ in \eqref{eq:condition_preconditioned}.}
If the sequence of coarse spaces $\{\VH\}_H$ satisfy the assumption that $H_D/q_D \lesssim \min_{D\in\calT_H} H_D/q_D$ for all $D\in\calT_H$, then the constant $\tilde{c}_0$ in the above proposition simplifies to
\begin{equation}
 \tilde{c}_0 \simeq 1+ \max_{D\in\calT_H}\left[ \frac{H_D}{q_D} \max_{K\in\calT_h(D)} \frac{p_K^2}{h_K} + \frac{H^3_D}{q_D^3} \max_{K\in\calT_h(D)} \frac{p_K^6}{h_K^3}\right].
\end{equation}
Moreover, if the sequences of meshes $\{\calT_H\}_H$ and $\{\calT_h\}_h$ are quasiuniform, and if the polynomial degrees are also quasiuniform in the sense that $q\coloneqq \max_D q_D \lesssim q_D$ for all $D \in \calT_H$ and $p\coloneqq\max_K p_K\lesssim p_K$ for all $K\in\calT_h$, then the condition number of the preconditioned system satisfies the bound
\begin{equation}\label{eq:predicted_condition_number}
 \kPA \lesssim \left(N_c+2\right)\left( 1+\frac{p^2 \, H }{q\, h}+\frac{p^6 \, H^3 }{q^3\, h^3 }\right).
\end{equation}
It is well-known that the above bound is optimal in terms of the powers of $H$ and $h$, see \cite{Brenner2005,Feng2005}. The numerical experiments of section~\ref{sec:numexp} show that the bound \eqref{eq:predicted_condition_number} is also sharp in terms of the orders of $p$ and $q$.
Choosing the coarse space such that $H\simeq h$ and $q\simeq p$ implies that $\kPA\lesssim p^3$, which shows that the preconditioner is robust with respect to $h$ but not with respect to $p$. The explicit dependence of our bound on $q$ shows nonetheless  a significant improvement over the condition number of order $p^8/h^4$ for the unpreconditioned matrix.
\cb{The preconditioner $\bo P$ can therefore be used to precondition the nonsymmetric systems~\eqref{eq:semi_newton_scheme} of the semismooth Newton method, where the convergence is guaranteed by Theorem~\ref{thm:GMRES_bound} in combination with \eqref{eq:predicted_condition_number}.}

\section{Approximation of discontinuous functions}\label{sec:approximability}
As explained in the introduction, the optimal bound for the condition numbers, as given by Theorem~\ref{thm:stable_decomposition}, rests upon the optimality of approximation properties between coarse and fine spaces. Therefore, in this section, we first determine how closely a function in $\Vh$ can be approximated by functions in $H^2(\Om)\cap H^1_0(\Om)$. This leads to an approximation result for functions in $\Vh$ by functions in $\VH$ that is of optimal order in both the coarse mesh size and polynomial degree. 

\subsection{Lifting operators}
Let $\Vh^{\dim}$ denote the space of $\dim$-dimensional vector fields with components in $\Vh$.
Let ${\bo r_h} \colon L^2(\calFhib) \tends \Vh^{\dim}$ and $r_h\colon L^2(\calF_h^i) \tends \Vh$ be defined by
\begin{align}
 \sum_{K\in\calT_h} \pair{{\bo r_h}(w)}{\bo v_h }_K &  = \sum_{F\in\calFhib} \pair{w}{ \lla \bo v_h \cdot n_F \rra}_F \qquad \forall \,\bo v_h \in \Vh^{\dim},\\
 \sum_{K\in\calT_h} \pair{r_h( w )}{v_h}_K & = \sum_{F\in\calF_h^i} \pair{ w }{ \lla v_h \rra}_F\qquad \forall\, v_h \in \Vh.
\end{align}
The following result is well-known; for instance, see \cite{Antonietti2011} for a proof.
\begin{lemma}\label{lem:lifting_stability}
 Let $\Om$ be a bounded Lipschitz domain and let $\{\calT_h\}_h$ be a shape-regular sequence of meshes satisfying \eqref{eq:card_F_bound}, \eqref{eq:c_h_bound} and \eqref{eq:c_p_bound}. Then, the lifting operators satisfy the following bounds:
\begin{subequations}
\begin{alignat}{2}
\sum_{K\in\calT_h} \norm{ {\bo r_h}(w) }_{L^2(K)}^2 &\lesssim \sum_{F\in\calFhib} \frac{\tilde{p}^2_F}{\tilde{h}_F} \norm{w}_{L^2(F)}^2& \qquad &\forall\, w \in L^2(\calFhib),\label{eq:jump_lift_stability}
\\ \sum_{K\in\calT_h}\norm{ r_h(w) }_{L^2(K)}^2 &\lesssim \sum_{F\in\calF_h^i} \frac{\tilde{p}^2_F}{\tilde{h}_F} \norm{w}_{L^2(F)}^2&\qquad &\forall\, w \in L^2(\calF_h^i).\label{eq:normal_lift_stability}
\end{alignat}
\end{subequations}
\end{lemma}
For $v_h\in\Vh$ and $\bo v_h\in\Vh^{\dim}$, define $G_h(v_h) \in \Vh^{\dim}$ and $D_h(\bo v_h)\in \Vh$ element-wise by 
\begin{subequations}
\begin{align}
G_h(v_h)|_K &\coloneqq \nabla v_h|_K - {\bo r_h}(\llb v_h \rrb)|_K, \label{eq:G_h_def}\\
D_h(\bo v_h)|_K &\coloneqq \Div \bo v_h|_K - r_h(\llb \bo v_h \cdot n_F\rrb)|_K,
\end{align}
\end{subequations}
for all $K\in \calT_h$. Observe that $D_h(\bo v_h)$ belongs to $L^2(\Om)$ for any $\bo v_h \in \Vh^{\dim}$.

\begin{lemma}\label{lem:rhj_estimates}
Let $\Om$ be a bounded Lipschitz polytopal domain, and let $\left\{\calT_h\right\}_h$ be a shape-regular sequence of meshes satisfying \eqref{eq:card_F_bound}, \eqref{eq:c_h_bound} and \eqref{eq:c_p_bound}. Let $\eta_F$ and $\mu_F$ satisfy \eqref{eq:eta_mu} for all $F\in\calFhib$. Then, for any $v_h \in \Vh$, we have
\begin{subequations}
\begin{gather}
 \sum_{K\in\calT_h} \frac{p^4_K}{h_K^2} \norm{{\bo r_h}(\llb v_h \rrb)}_{L^2(K)}^2 \lesssim \absj{v_h}^2, \label{eq:rhj_estimate_2}
\\ \sum_{K\in\calT_h} \abs{ {\bo r_h}(\llb v_h \rrb)}_{H^1(K)}^2 + \sum_{F\in\calF_h^i} \mu_F \norm{ \llb {\bo r_h}(\llb v_h \rrb) \cdot n_F \rrb }_{L^2(F)}^2   \lesssim \absj{v_h}^2.\label{eq:rhj_estimate_3}
\end{gather}
\end{subequations}
\end{lemma}
\begin{proof}
 The definition of the lifting operator gives
\begin{multline*}
 \sum_{K\in\calT_h} \frac{p^4_K}{h_K^2} \norm{{\bo r_h}(\llb v_h \rrb)}_{L^2(K)}^2  = \sum_{F\in\calFhib} \int_F \llb v_h \rrb \lla \frac{p^4}{h^2} {\bo r_h}(\llb v_h \rrb)\cdot n_F \rra \,\d s
\\ \lesssim \left( \sum_{F\in\calFhib} \frac{\tilde{h}_F^3}{\tilde{p}_F^6} \frac{\tilde{p}_F^8}{\tilde{h}_F^4} \norm{{\bo r_h}(\llb v_h \rrb)}_{L^2(F)}^2 \right)^{\frac{1}{2}} \absj{v_h}.
\end{multline*}
The trace and inverse inequalities then yield
\begin{equation*}
  \sum_{K\in\calT_h} \frac{p^4_K}{h_K^2} \norm{{\bo r_h}(\llb v_h \rrb)}_{L^2(K)}^2 \lesssim \left( \sum_{K\in\calT_h} \frac{p_K^4}{h_K^2} \norm{{\bo r_h}(\llb v_h \rrb)}_{L^2(K)}^2 \right)^{\frac{1}{2}} \absj{v_h},
\end{equation*}
which implies \eqref{eq:rhj_estimate_2}. The bound \eqref{eq:rhj_estimate_3} then follows from \eqref{eq:rhj_estimate_2} as a result of the trace and inverse inequalities.
\end{proof}

\begin{corollary}\label{cor:G_h_stability}
 Under the hypotheses of Lemma~\ref{lem:rhj_estimates}, every $v_h \in \Vh$ satisfies
\begin{equation}\label{eq:G_h_stability}
 \sum_{K\in\calT_h} \abs{G_h(v_h)}_{H^1(K)}^2 + \sum_{F\in\calF_h^i} \mu_F \norm{ \llb G_h(v_h) \cdot n_F \rrb}_{L^2(F)}^2 \lesssim \normh{v_h}^2.
\end{equation}
We also have $\norm{D_h(G_h(v_h))}_{L^2(\Om)} \lesssim \normh{v_h}$ for every $v_h \in\Vh$.
\end{corollary}
\begin{proof}
Inequality \eqref{eq:G_h_stability} is an easy consequence of the definition of $G_h$ in \eqref{eq:G_h_def} and of Lemma~\ref{lem:rhj_estimates}.
For $v_h\in\Vh$ and $K\in \calT_h$, we have
\begin{equation}\label{eq:G_h_stability_1}
 D_h(G_h(v_h))|_K =  \left[\Delta v_h - \Div {\bo r_h}(\llb v_h \rrb) - r_h(\llb \nabla v_h \cdot n_F \rrb) + r_h( \llb {\bo r_h}(\llb v_h \rrb) \cdot n_F\rrb)\right]|_K.
\end{equation}
In view of \eqref{eq:normal_lift_stability}, it is apparent that the global $L^2$-norms over $\Om$ of the first and third terms on the right-hand side of \eqref{eq:G_h_stability_1} are bounded by $\normh{v_h}$, whilst the bounds on the $L^2$-norms of the second and fourth terms follow from \eqref{eq:rhj_estimate_3}.
\end{proof}
\subsection{Approximation by $H^2$-regular functions}
The first step towards the aforementioned approximation result is to consider the discrete analogue of the orthogonality of Helmholtz decompositions.
 \cb{In this section, we shall view the element-wise gradient of a function $v_h \in \Vh$ as an element of $L^2(\Om)^\dim$, and thus we denote it by $\nabla_h v_h$.}

\begin{lemma}\label{lem:discrete_grad_curl_orthogonal}
Let $\Om\subset \R^{\dim}$, $\dim\in\{2,3\}$, be a bounded Lipschitz polytopal domain, and let $\left\{\calT_h\right\}_h$ be a shape-regular sequence of meshes satisfying \eqref{eq:card_F_bound}, \eqref{eq:c_h_bound} and \eqref{eq:c_p_bound}. If $\mu_F$ and $\eta_F$ satisfy \eqref{eq:eta_mu} for every face $F\in\calFhib$, then, for any $v_h \in \Vh$ and any $\psi \in H^1(\Om)^{2\dim-3}$, we have
\begin{equation}\label{eq:discrete_grad_curl_orthogonal_h1}
 \Labs{ \int_\Om G_h(v_h) \cdot \Curl \psi \, \d x } + \Labs{ \int_\Om \nabla_h v_h \cdot \Curl \psi \, \d x } \lesssim \max_{K\in\calT_h} \frac{h_K}{p^{3/2}_K} \absj{v_h}\norm{\psi}_{H^1(\Om)}.
\end{equation}
\end{lemma}
\begin{proof}
It follows from \eqref{eq:rhj_estimate_2} that $\norm{\nabla_h v_h-G_h(v_h)}_{L^2(\Om)} \lesssim \max_K h_K/p_K^2 \absj{v_h}$, so it is enough to show that \eqref{eq:discrete_grad_curl_orthogonal_h1} is satisfied by $G_h(v_h)$.
Consider momentarily $\psi \in H^2(\Om)^{2\dim-3}$; then, integration by parts yields 
\[
 \int_\Om G_h(v_h) \cdot \Curl \psi \, \d x = \sum_{F\in\calFhib} \pair{\llb v_h \rrb}{\lla \Curl \psi \cdot n_F \rra}_F - \sum_{K\in\calT_h}\pair{{\bo r_h}(\llb v_h \rrb)}{\Curl \psi}_K.
\]
Therefore, the definitions of the lifting operators $\bo r_h$ and $r_h$ imply that
\begin{multline*}
 \int_\Om G_h(v_h) \cdot \Curl \psi \, \d x = \sum_{F\in\calFhib} \pair{\llb v_h \rrb}{\lla \Curl (\psi-\Pi_h\psi) \cdot n_F \rra}_F\\ - \sum_{K\in\calT_h}\pair{{\bo r_h}(\llb v_h \rrb)}{\Curl (\psi-\Pi_h\psi)}_K.
\end{multline*}
Thus, if $\psi \in H^2(\Om)^{2\dim-3}$, it is seen from the approximation bounds of \eqref{eq:approximation_bound} and from the lifting bound \eqref{eq:rhj_estimate_2} that
\begin{equation}\label{eq:discrete_grad_curl_orthogonal_h2}
 \Labs{\int_\Om G_h(v_h) \cdot \Curl \psi \, \d x} \lesssim \max_{K\in\calT_h} \frac{h_K^2}{p_K^3} \absj{v_h} \norm{\psi}_{H^2(\Om)}.
\end{equation}
Now, let $\psi \in H^1(\Om)^{2\dim-3}$. We apply \cite[Thm.~5.33]{Adams2003} to the components of $\psi$: for each $\eps>0$, there exists a $\psi_\eps \in C^{\infty}(\R^{\dim})^{2\dim-3}$ such that
\begin{subequations}\label{eq:psi_eps_approx}
\begin{gather}
 \norm{\psi-\psi_\eps}_{L^2(\Om)} + \eps \norm{\psi - \psi_\eps}_{H^1(\Om)} \lesssim \eps \abs{\psi}_{H^1(\Om)},\label{eq:psi_eps_direct}
 \\ \norm{\psi_\eps}_{H^2(\Om)} \lesssim \eps^{-1} \norm{\psi}_{H^1(\Om)},\label{eq:psi_eps_inverse}
\end{gather}
\end{subequations}
where, importantly, the constants in \eqref{eq:psi_eps_approx} do not depend on $\eps$.
Define $\phi_\eps\coloneqq\psi-\psi_\eps$, so that
\[
\int_\Om G_h(v_h) \cdot \Curl \psi \, \d x  = \int_\Om G_h(v_h) \cdot \Curl \psi_\eps \, \d x + \int_\Om G_h(v_h) \cdot \Curl \phi_\eps \, \d x.
\]
The bounds \eqref{eq:discrete_grad_curl_orthogonal_h2} and \eqref{eq:psi_eps_inverse} show that 
\begin{equation}\label{eq:grad_curl_1}
 \Labs{\int_\Om G_h(v_h) \cdot \Curl \psi_\eps \, \d x} \lesssim \eps^{-1} \max_{K\in\calT_h} \frac{h_K^2}{p_K^3} \absj{v_h} \norm{\psi}_{H^1(\Om)}.
\end{equation}
Integration by parts yields
\[
\int_\Om G_h(v_h) \cdot \Curl \phi_\eps \, \d x = \sum_{F\in\calFhib} \pair{\llb \nabla v_h \times n_F \rrb}{\phi_\eps}_F-\sum_{K\in\calT_h} \pair{{\bo r_h}(\llb v_h \rrb)}{\Curl \phi_\eps}_K.
\]
Lemma~\ref{lem:rhj_estimates} and \eqref{eq:psi_eps_direct} imply that
\begin{equation}\label{eq:grad_curl_3}
 \sum_{K\in\calT_h} \abs{\pair{{\bo r_h}(\llb v_h \rrb)}{\Curl \phi_\eps}_K} \lesssim \max_{K\in\calT_h}\frac{h_K}{p_K^2} \absj{v_h} \norm{\psi}_{H^1(\Om)}.
\end{equation}
Recall the continuous trace inequality \cite{Monk1999}: for an element $K$ and a face $F\subset \p K$,
\[
   \norm{\phi_\eps}_{L^2(F)}^2 \lesssim \abs{\phi_\eps}_{H^1(K)}\norm{\phi_\eps}_{L^2(K)} + \frac{1}{h_K}\norm{\phi_\eps}_{L^2(K)}^2
 \lesssim \frac{h_K}{p_K^2} \abs{\phi_\eps}_{H^1(K)}^2+\frac{p_K^2}{h_K}\norm{\phi_\eps}_{L^2(K)}^2.
\]
Therefore, the fact that $\mu_F = \cmu \,\tilde{p}_F^2/\tilde{h}_F$ leads to
\begin{multline*}
  \sum_{F\in\calFhib}\abs{ \pair{\llb \nabla v_h \times n_F \rrb}{\phi_\eps}_F } \lesssim  \left( \sum_{K\in\calT_h} \left[\frac{h_K^2}{p_K^4} \abs{\phi_\eps}_{H^1(K)}^2 + \norm{\phi_\eps}_{L^2(K)}^2 \right] \right)^{\frac{1}{2}} \absj{v_h}
\\ \lesssim \left( \max_{K\in\calT_h} \frac{h_K}{p_K^2} \abs{\phi_\eps}_{H^1(\Om)} + \norm{\phi_\eps}_{L^2(\Om)} \right) \absj{v_h},
\end{multline*}
where we have used the identity $\abs{\llb \nabla v_h \times n_F \rrb}=\abs{\llb \nablaT v_h \rrb}$ for each face $F$, because $\nablaT v_h$ is the component of $\nabla v_h$ that is orthogonal to $n_F$.
Therefore, we deduce from \eqref{eq:psi_eps_direct} and \eqref{eq:grad_curl_3} that
\begin{equation}\label{eq:grad_curl_2}
 \Labs{\int_\Om G_h(v_h) \cdot \Curl \phi_\eps \, \d x} \lesssim \left( \max_{K\in\calT_h} \frac{h_K}{p_K^2} + \eps  \right) \absj{v_h}\norm{\psi}_{H^1(\Om)}.
\end{equation}
Combining \eqref{eq:grad_curl_1} and \eqref{eq:grad_curl_2} yields
\[
 \Labs{ \int_\Om G_h(v_h) \cdot \Curl \psi \, \d x } \lesssim \left( \eps^{-1} \max_{K\in\calT_h} \frac{h^2_K}{p^3_K} + \max_{K\in\calT_h}\frac{h_K}{p_K^2} + \eps \right) \absj{v_h} \norm{\psi}_{H^1(\Om)}.
\]
The bound \eqref{eq:discrete_grad_curl_orthogonal_h1} is then obtained by taking $\eps\coloneqq \max_{K\in\calT_h} h_K/p_K^{3/2}$.
\end{proof}

\begin{theorem}\label{thm:dg_h2_approximation}
Let $\Om\subset \R^{\dim}$, $\dim\in\{2,3\}$, be a bounded convex polytopal domain, and let $\left\{\calT_h\right\}_h$ be a shape-regular sequence of meshes satisfying \eqref{eq:card_F_bound}, \eqref{eq:c_h_bound} and \eqref{eq:c_p_bound}. For a given $v_h \in \Vh$, let $v_{(h)}\in H^2(\Om)\cap H^1_0(\Om)$ be the unique solution of the boundary-value problem
\begin{subequations}\label{eq:dg_h2_approximation_def}
 \begin{alignat}{2}
\Delta v_{(h)}& = D_h(G_h(v_h)) & &\quad\text{in }\Om,
\\ v_{(h)} &= 0 & &\quad\text{on }\DO.
\end{alignat}
\end{subequations}
Then, the approximation $v_{(h)}$ to $v_h$ satisfies
\begin{subequations}
\begin{gather}
 \norm{v_h - v_{(h)}}_{L^2(\Om)} + \max_{K\in\calT_h} \frac{h_K}{p_K}\norm{v_h - v_{(h)}}_{H^1(\Om;\calT_h)} \lesssim \max_{K\in\calT_h} \frac{h_K^2}{p_K^2}\absj{v_h},\label{eq:dg_h2_approximation_1}
\\ \norm{v_{(h)}}_{H^2(\Om)} \lesssim \normh{v_h}.\label{eq:dg_h2_approximation_bound_2}
\end{gather}
\end{subequations}
\end{theorem}
\vspace{-10pt}
\begin{remark}
 The above result is nearly optimal in the sense that only the jump seminorm $\absj{v_h}$ appears on the right-hand side of the error bound \eqref{eq:dg_h2_approximation_1}, and that the correct orders of convergence are established.
\end{remark}

\begin{proof}
Note that convexity of $\Om$ implies that $v_{(h)}$ is well-defined, see \cite{Grisvard2011}, and that \eqref{eq:dg_h2_approximation_bound_2} holds as a result of Corollary~\ref{cor:G_h_stability}.  
First, we show that for any $p \in H^k(\Om)\cap H^1_0(\Om)$, $k\in\{1,2\}$, we have
\begin{equation}\label{eq:dg_h2_approximation_2}
 \Labs{ \int_\Om \left( \nabla v_{(h)}-G_h(v_h) \right)\cdot \nabla p \, \d x } \lesssim \max_{K\in\calT_h} \frac{h^k_K}{p_K^k} \absj{v_h} \norm{p}_{H^k(\Om)}.
\end{equation}
Indeed, since $v_{(h)}$ solves \eqref{eq:dg_h2_approximation_def}, integration by parts yields
\begin{multline*}
 \int_\Om \left(\nabla v_{(h)}-G_h(v_h)\right)\cdot \nabla p \, \d x  = \sum_{K\in\calT_h} \pair{ r_h(\llb G_h(v_h)\cdot n_F \rrb)}{p}_K \\ -\sum_{F\in\calF_h^i} \pair{ \llb G_h(v_h)\cdot n_F \rrb}{\lla p \rra }_F.
\end{multline*}
Then, the definition of the lifting operator gives
\begin{multline}\label{eq:dg_h2_approximation_4}
\int_\Om \left(\nabla v_{(h)}-G_h(v_h)\right)\cdot \nabla p \, \d x   = \sum_{K\in\calT_h} \pair{ r_h(\llb G_h(v_h)\cdot n_F \rrb)}{p-\Pi_h p}_K \\ -\sum_{F\in\calF_h^i}\pair{ \llb G_h(v_h)\cdot n_F \rrb}{\lla p - \Pi_h p\rra }_F.
\end{multline}
Recalling that $p_K\geq1$ for each element $K$, it is then seen that \eqref{eq:dg_h2_approximation_2} follows from Corollary~\ref{cor:G_h_stability} and from the approximation bounds \eqref{eq:approximation_bound}.

The remainder of the proof makes use of Helmholtz decompositions of vector fields \cite{Girault1986}: for any $\bo v \in L^2(\Om)^{\dim}$, there exists $p\in H^1_0(\Om)$ and $\psi \in H^1(\Om)^{2\dim-3}$, such that $\bo v = \nabla p + \Curl \psi$ in $\Om$. Indeed, $p\in H^1_0(\Om)$ is defined by 
\[
 \int_\Om \nabla p \cdot \nabla q \, \d x = \int_\Om \bo v \cdot \nabla q \, \d x \qquad \forall\,q\in H^1_0(\Om).
\]
Then, $\bo v - \nabla p$ is divergence free, thus $\pair{\left(\bo v - \nabla p\right)\cdot \bo n}{1}_{\DO}=0$, where $\bo n$ is the unit outward normal on $\DO$. Since the convex domain $\Om$ has a connected boundary, it follows from \cite[Thms.~3.1 \&~3.4 pp.~37--45]{Girault1986} that there exists a $\psi \in H^1(\Om)^{2\dim-3}$ such that $\bo v = \nabla p + \Curl \psi$.
Moreover, $\psi$ may be chosen so that $\norm{p}_{H^1(\Om)}+\norm{\psi}_{H^1(\Om)} \lesssim \norm{ \bo v}_{L^2(\Om)}$ for some constant independent of $\bo v$. This is a consequence of the Open Mapping Theorem and the facts that $ \mathcal{V}\coloneqq\{ \bo v \in L^2(\Om)^{\dim}\colon \Div \bo v =0 \}$ is a closed subspace of $L^2(\Om)^{\dim}$, and that the mapping $ \psi \mapsto \Curl \psi$ is a surjective bounded linear mapping from $H^1(\Om)^{2\dim-3}$ to $\mathcal{V}$.

Now, observe that $\norm{\nabla v_h - G_h(v_h)}_{L^2(\Om)} \lesssim \max_{K\in\calT_h} h_K/p_K^2 \, \absj{v_h}$ by \eqref{eq:rhj_estimate_2}, so it is enough to consider the error between $G_h(v_h)$ and $\nabla v_{(h)}$ to bound $\abs{v_h - v_{(h)}}_{H^1(\Om;\calT_h)}$.
Let $p \in H^1_0(\Om)$ and $\psi \in H^1(\Om)^{2\dim-3}$ satisfy $ \nabla v_{(h)} - G_h(v_h) = \nabla p + \Curl \psi$, with $\norm{p}_{H^1(\Om)}+\norm{\psi}_{H^1(\Om)} \lesssim \norm{ \nabla v_{(h)}-G_h(v_h)}_{L^2(\Om)}$. Then, noting that $\nabla v_{(h)}$ and $\Curl \psi$ are orthogonal, it is deduced that
\begin{equation}\label{eq:dg_h2_approximation_3}
 \norm{  \nabla v_{(h)}-G_h(v_h) }_{L^2(\Om)}^2 = \int_\Om \left(\nabla v_{(h)}-G_h(v_h)\right)\cdot \nabla p \, \d x - \int_{\Om} G_h(v_h) \cdot \Curl \psi \, \d x.
\end{equation}
Inequality \eqref{eq:dg_h2_approximation_2} and the bound $\norm{p}_{H^1(\Om)} \lesssim \norm{\nabla v_{(h)}-G_h(v_h)}_{L^2(\Om)}$ give
\[
 \Labs{\int_\Om \left(\nabla v_{(h)}-G_h(v_h)\right)\cdot \nabla p \, \d x } \lesssim \max_{K\in\calT_h}\frac{h_K}{p_K} \absj{v_h} \norm{\nabla v_{(h)}-G_h(v_h)}_{L^2(\Om)}.
\]
The bounds of Lemma~\ref{lem:discrete_grad_curl_orthogonal} show that
\[
 \Labs{\int_{\Om} G_h(v_h) \cdot \Curl \psi \, \d x} \lesssim \max_{K\in\calT_h} \frac{h_K}{p_K^{3/2}} \absj{v_h} \norm{\nabla v_{(h)}-G_h(v_h)}_{L^2(\Om)}.
\]
Therefore, equation \eqref{eq:dg_h2_approximation_3} and the above bounds yield
\begin{equation}\label{eq:dg_h2_approximation_6}
\norm{\nabla v_{(h)}-G_h(v_h)}_{L^2(\Om)} \lesssim \max_{K\in\calT_h} \frac{h_K}{p_K} \absj{v_h}.
\end{equation}
We now consider the error $\norm{v_h - v_{(h)}}_{L^2(\Om)}$. Since $\Om$ is convex, there is a unique $z\in H^2(\Om)\cap H^1_0(\Om)$ that solves $-\Delta z = v_h - v_{(h)}$ in $\Om$, with $\norm{z}_{H^2(\Om)}\lesssim \norm{v_h-v_{(h)}}_{L^2(\Om)}$. Then, it is found that
\begin{multline*}
 \norm{v_h - v_{(h)}}_{L^2(\Om)}^2 = \int_\Om \left(G_h(v_h)-\nabla v_{(h)}\right)\cdot \nabla z\, \d x
\\ + \sum_{K\in\calT_h} \pair{{\bo r_h}(\llb v_h \rrb)}{\nabla z}_K - \sum_{F\in\calFhib} \pair{\llb v_h \rrb}{\lla \nabla z \cdot n_F \rra}_F.
\end{multline*}
Applying the bound \eqref{eq:dg_h2_approximation_2} to $z \in H^2(\Om)\cap H^1_0(\Om)$ gives 
\[
 \Labs{\int_\Om \left(G_h(v_h)-\nabla v_{(h)}\right)\cdot \nabla z \,\d x} \lesssim \max_{K\in\calT_h} \frac{h^2_K}{p_K^2} \absj{v_h} \norm{v_h-v_{(h)}}_{L^2(\Om)}.
\]
Also, it is found that 
\begin{multline*}
 \sum_{K\in\calT_h} \pair{{\bo r_h}(\llb v_h \rrb)}{\nabla z}_K - \sum_{F\in\calFhib} \pair{\llb v_h \rrb}{\lla \nabla z \cdot n_F \rra}_F
\\ = \sum_{K\in\calT_h} \pair{{\bo r_h}(\llb v_h \rrb)}{\nabla (z-\Pi_h z)}_K - \sum_{F\in\calFhib} \pair{\llb v_h \rrb}{\lla \nabla (z-\Pi_h z) \cdot n_F \rra}_F,
\end{multline*}
which is bounded by $\max_K h_K^2/p_K^{3} \absj{v_h} \norm{v_h-v_{(h)}}_{L^2(\Om)}$. Thus, we have shown that
\begin{equation}\label{eq:dg_h2_approximation_5}
 \norm{v_h - v_{(h)}}_{L^2(\Om)} \lesssim \max_{K\in\calT_h} \frac{h_K^2}{p_K^2} \absj{v_h}.
\end{equation}
The bounds \eqref{eq:dg_h2_approximation_6} and \eqref{eq:dg_h2_approximation_5} imply \eqref{eq:dg_h2_approximation_1}.
\end{proof}

\subsection{Approximation by coarse grid functions}
Theorem~\ref{thm:dg_h2_approximation} leads to the following approximation result between coarse and fine spaces.
\begin{theorem}\label{thm:dg_coarse_approximation}
 Let $\Om\subset \R^{\dim}$, $\dim\in\{2,3\}$, be a bounded convex polytopal domain, and let $\{\calT_H\}_H$ and $\{\calT_h\}_h$ be nested shape-regular sequences of meshes satisfying \eqref{eq:card_F_bound}, \eqref{eq:c_h_bound} and \eqref{eq:c_p_bound}. Then, for any $v_h \in \Vh$, there exists a $v_H\in\VH$, such that
\begin{subequations}
\begin{gather}
\norm{v_h-v_H}_{H^k(\Om;\calT_h)} \lesssim \left(\max_{D\in\calT_H}\frac{H_D}{q_D}\right)^{2-k} \,\normh{v_h}, \qquad k\in\{0,1,2\}. \label{eq:dg_coarse_approximation_1}
\\ \normh{v_H}^2 \lesssim \left(1 + \max_{D\in\calT_H} \left[\frac{H_D}{q_D} \max_{K\in\calT_h(D)}\frac{p_K^2}{h_K} + \frac{H^3_D}{q_D^3} \max_{K\in\calT_h(D)}\frac{p_K^6}{h_K^3} \right]\right)  \normh{v_h}^2.\label{eq:dg_coarse_approximation_2}
\end{gather}
\end{subequations}
\end{theorem}
\begin{proof}
Let $v_{(h)}\in H^2(\Om)\cap H^1_0(\Om)$ be the approximation to $v_h$ considered in Theorem~\ref{thm:dg_h2_approximation}. Let $v_H \in \VH$ be the projection $\Pi_H v_{(h)}$. Since $\max_{K\in\calT_h}  h_K /p_K \leq \max_{D\in\calT_H} H_D/q_D$, it is seen that \eqref{eq:dg_coarse_approximation_1} follows easily from the triangle inequality in conjunction with \eqref{eq:dg_h2_approximation_bound_2} and the approximation properties of $v_H$. In particular, it follows from $v_H = \Pi_H v_{(h)}$ that $\norm{v_H}_{H^2(\Om;\calT_h)} \lesssim \norm{v_{(h)}}_{H^2(\Om)}$, and since Theorem~\ref{thm:dg_h2_approximation} implies that $\norm{v_{(h)}}_{H^2(\Om)} \lesssim \normh{v_h}$, we obtain $\norm{v_H}_{H^2(\Om;\calT_h)} \lesssim \normh{v_h}$.

It remains to show \eqref{eq:dg_coarse_approximation_2} by bounding the jump seminorm of $v_H$ as follows. If the face $F\in \calF_h^i(D)$ for $D\in\calT_H$, then the jumps of $v_H$ and its first derivatives vanish because $v_H$ is a polynomial over $D$. Since $v_{(h)} \in H^2(\Om)\cap H^1_0(\Om)$, $\llb v_H \rrb = \llb v_H - v_{(h)} \rrb $ and $\llb \nablaT v_H \rrb = \llb \nablaT (v_H-v_{(h)}) \rrb $ for each face $F\in\calFhib(\p D)$, whilst $\llb \nabla v_H\cdot n_F \rrb = \llb \nabla (v_H-v_{(h)})\cdot n_F \rrb$ for each face $F\in\calF_h^i(\p D)$. Therefore, it is deduced from the mesh assumptions on $\calT_h$ and $\calT_H$ that
\[\begin{split}
 \sum_{F\in\calFhib} \eta_F \norm{\llb v_H \rrb}_{L^2(F)}^2  & \leq \sum_{D\in\calT_H} \sum_{F\in\calFhib(\p D)}\eta_F \norm{\llb v_H - v_{(h)} \rrb}_{L^2(F)}^2
\\ &\lesssim \sum_{D\in\calT_H}  \max_{K\in \calT_h(D)} \frac{p_K^6}{h_K^3} \norm{ v_H-v_{(h)} }_{L^2(\p D)}^2
\\ &\lesssim \max_{D\in\calT_H} \left[\frac{H_D^3}{q_D^3} \max_{K\in \calT_h(D)} \frac{p_K^6}{h_K^3}\right] \norm{v_{(h)}}_{H^2(\Om)}^2.
\end{split}\]
Similar bounds also yield
\begin{multline*}
 \sum_{F\in\calFhib} \mu_F \norm{\llb \nablaT v_H \rrb}_{L^2(F)}^2 + \sum_{F\in\calF_h^i} \mu_F \norm{\llb \nabla v_H \cdot n_F \rrb}_{L^2(F)}^2
\\ \lesssim \max_{D\in\calT_H}\left[\frac{H_D}{q_D}\max_{K\in\calT_h(D)}\frac{p_K^2}{h_K}\right] \norm{v_{(h)}}_{H^2(\Om)}^2.
\end{multline*}
Since $\norm{v_{(h)}}_{H^2(\Om)} \lesssim \normh{v_h}$, the proof of \eqref{eq:dg_coarse_approximation_2} is complete.
\end{proof}

Previous results on the approximation of fine mesh functions by coarse mesh functions typically involved lower-order projection operators, which were therefore suboptimal in terms of $q$ in bounds such as \eqref{eq:dg_coarse_approximation_1}. The original result of an approximation with optimal orders in both $H$ and $q$ of Theorem~\ref{thm:dg_coarse_approximation} enables the sharp analysis of the nonoverlapping domain decomposition preconditioners in the next section.

\section{Stable~decomposition property}\label{sec:stable_decomposition}

The following lemma, due to Feng and Karakashian in \cite{Feng2001}, provides a trace inequality for the boundaries $\p D$ of elements $D\in\calT_H$. However, the inequality is not written there in the form that is required for our purposes. So, we present again the proof, with some variations from the arguments in \cite{Feng2001}.
\begin{lemma}\label{lem:discrete_trace_inequality}
 Let $\{\calT_H\}_H$ and $\{\calT_h\}_h$ be shape-regular sequences of nested simplicial or parallelepipedal meshes satisfying the conditions \eqref{eq:card_F_bound} and \eqref{eq:c_h_bound}, and let $\bo p$ satisfy \eqref{eq:c_p_bound}. Let $v \in L^2(D)$ belong to $\calP_{p_K}(K)$ for each $K\subset D$. Then, we have
\begin{multline}\label{eq:discrete_trace_inequality_0}
 \norm{v}_{L^2(\p D)}^2 \lesssim \sum_{K\in \calT_{h}(D)} \abs{v}_{H^1(K)}\norm{v}_{L^2(K)} + \frac{1}{H_D} \norm{v}_{L^2(D)}^2
\\ + \left(\sum_{F\in \calF_{h}^i(D)} \frac{\tilde{p}_F^2}{\tilde{h}_F} \norm{\llb v \rrb}_{L^2(F)}^2 \right)^{\frac{1}{2}}\,\norm{v}_{L^2(D)}.
\end{multline}
\end{lemma}
\begin{proof}
As shown in \cite{Feng2001}, since each element $D\in\calT_H$ is an affine image of a convex reference element, it follows that there is a point $x_0 \in D$, such that $(x-x_0)\cdot n_{\p D} \gtrsim H_D$ for each $x\in \p D$, where $n_{\p D}$ is the unit outward normal vector to $\p D$. Therefore,
\begin{equation}\label{eq:discrete_trace_inequality_1}
 \norm{v}_{L^2(\p D)}^2 \lesssim \frac{1}{H_D} \int_{\p D} \abs{v}^2 \left(x-x_0\right)\cdot n_{\p D} \, \d s.
\end{equation}
Integration by parts shows that
\begin{multline*}
 \int_{\p D} \abs{v}^2 \left(x-x_0\right)\cdot n_{\p D} \, \d s = \sum_{K\in\calT_h(D)} \int_{K} \left[\Div\left(x-x_0\right) \abs{v}^2 + 2 v \,\nabla v \cdot \left(x-x_0\right)\right] \d x  \\ - \sum_{F\in\calF_h^i(D)} \pair{\llb v^2 \rrb}{\lla \left(x-x_0\right)\cdot n_F \rra}_F.
\end{multline*}
Since $\llb v^2 \rrb = 2 \llb v \rrb \lla v \rra$, it is found that
\begin{multline*}
 \int_{\p D} \abs{v}^2 \left(x-x_0\right)\cdot n_{\p D} \, \d s \lesssim H_D \sum_{K\in\calT_h(D)} \abs{v}_{H^1(K)}\norm{v}_{L^2(K)} + \norm{v}_{L^2(D)}^2 
\\ + H_D \left(\sum_{F\in\calF_h^i(D)} \frac{\tilde{p}_F^2}{\tilde{h}_F} \norm{\llb v \rrb }_{L^2(F)}^2 \right)^{\frac{1}{2}}\left( \sum_{F\in\calF_h^i(D)}  \frac{\tilde{h}_F}{\tilde{p}_F^2}\norm{\lla v \rra}_{L^2(F)}^2 \right)^{\frac{1}{2}}.
\end{multline*}
The inverse and trace inequalities imply that
\[
\sum_{F\in\calF_h^i(D)}  \frac{\tilde{h}_F}{\tilde{p}_F^2}\norm{\lla v \rra}_{L^2(F)}^2 \lesssim \norm{v}_{L^2(D)}^2. 
\]
Therefore, \eqref{eq:discrete_trace_inequality_0} follows from \eqref{eq:discrete_trace_inequality_1} and the above bounds.
\end{proof}

Equipped with the approximation result of Theorem~\ref{thm:dg_coarse_approximation}, it is now possible to prove Theorem~\ref{thm:stable_decomposition} using a similar approach to \cite{Antonietti2011,Feng2001,Feng2005}.

\subsection{Proof of Theorem~\ref{thm:stable_decomposition}}
Let $v_H$ be given as in Theorem~\ref{thm:dg_coarse_approximation}, set $v_0\coloneqq v_H$, and denote by $v_i \in \Vh^i$ the restriction of $ v_h - v_H$ to $\Om_i$, $1\leq i \leq N$.
Then, we have
\begin{equation}\label{eq:stable_decomposition_1}
 \sum_{i=0}^N a_h^i(v_i,v_i) = a_h(v_H,v_H) + a_h(v_h-v_H,v_h-v_H)-\sum_{\substack{i,\,j=1\\i\neq j}}^N a_h(I_i v_i,I_j v_j).
\end{equation}
Observe that the constant appearing on the right-hand side of \eqref{eq:dg_coarse_approximation_2} can be bounded in terms of $\tilde{c}_0$, which was defined in \eqref{eq:c_0_constant}.
So, Theorem~\ref{thm:dg_coarse_approximation} and Lemma~\ref{lem:a_h_coercivity} imply
\begin{subequations}\label{eq:stable_decomposition_2}
\begin{gather}
 a_h(v_H,v_H)\lesssim \normh{v_H}^2  \lesssim  \tilde{c}_0 \, a_h(v_h,v_h),
\\ a_h(v_h-v_H,v_h-v_H)  \lesssim \normh{v_h}^2+\normh{v_H}^2 \lesssim \tilde{c}_0 \, a_h(v_h,v_h).
\end{gather}
\end{subequations}
It remains to bound the last term in \eqref{eq:stable_decomposition_1} for the interface flux and jump terms at the boundaries of the subdomains of $\calT_S$. Expanding this term leads to
\begin{equation}
\sum_{\substack{i,\,j=1\\i\neq j}}^N \abs{a_h(I_i v_i,I_j v_j)} \leq \sum_{k=1}^5 E_k,
\end{equation}
where the quantities $E_k$ are defined by
\begin{subequations}\label{eq:E_definitions}
\begin{align}
 E_1 & \coloneqq \sum_{\substack{i,\,j=1\\i\neq j}}^N \sum_{\substack{F\in\calFhi \\ F\subset \p \Om_i \cap \p \Om_j}}
 \eta_F\,\abs{ \pair{\eval{(v_h-v_H)}{\Om_i}}{\eval{(v_h-v_H)}{\Om_j}}_F } ,
\\ E_2 & \coloneqq  \sum_{\substack{i,\,j=1\\i\neq j}}^N \sum_{\substack{F\in\calFhi \\ F\subset \p \Om_i \cap \p \Om_j}} \mu_F\,  \abs{\pair{\eval{\nablaT (v_h-v_H)}{\Om_i}}{ \eval{\nablaT(v_h-v_H)}{\Om_j}}_F} ,
\\ E_3 &\coloneqq \sum_{\substack{i,\,j=1\\i\neq j}}^N \sum_{\substack{F\in\calF_h^i \\ F\subset \p \Om_i \cap \p \Om_j}} \mu_F\,\abs{ \pair{\eval{\nabla (v_h-v_H)}{\Om_i}\cdot n_F}{ \eval{\nabla(v_h-v_H)}{\Om_j}\cdot n_F}_F},
\\ E_4 &\coloneqq  \sum_{\substack{i,\,j=1\\i\neq j}}^N \sum_{\substack{F\in\calFhi \\ F\subset \p \Om_i \cap \p \Om_j}}\abs{\pair{\divT\nablaT \eval{(v_h-v_H)}{\Om_i}}{\eval{\nabla(v_h-v_H)}{\Om_j}\cdot n_F}_F},
\\ E_5 &\coloneqq  \sum_{\substack{i,\,j=1\\i\neq j}}^N \sum_{\substack{F\in\calFhi \\ F\subset \p \Om_i \cap \p \Om_j}}\abs{\pair{\eval{\nablaT(\nabla(v_h-v_H)}{\Om_i}\cdot n_F)}{\eval{\nablaT(v_h-v_H)}{\Om_j}}_F}.
\end{align}
\end{subequations}
Note that in \eqref{eq:E_definitions}, we have made use of the symmetry of the sum over $i$, $j$, $i\neq j$, and the fact that any face $F\subset\p\Om_i\cap\p\Om_j$ must be an interior face.

Defining $\eta_D \coloneqq  \max_{K\in\calT_h(D)} p_K^6/h_K^3$ for each $D\in\calT_H$, the hypotheses \eqref{eq:c_h_bound} and \eqref{eq:c_p_bound} and the nestedness of the meshes imply that
\[
 E_1 \lesssim \sum_{D\in\calT_H} \eta_D \norm{v_h-v_H}_{L^2(\p D)}^2.
\]
Therefore, using the trace inequality of Lemma~\ref{lem:discrete_trace_inequality}, we find that
\begin{multline*}
 E_1 \lesssim \sum_{D\in\calT_H} \eta_D
\Biggl[ \frac{H_D}{q_D}  \sum_{K\in\calT_h(D)} \abs{v_h-v_H}_{H^1(K)}^2 + \frac{H_D}{q_D}  \sum_{F\in\calF_h^i(D)} \frac{\tilde{p}_F^2}{\tilde{h}_F} \norm{\llb v_h \rrb}_{L^2(F)}^2   \\
 + \frac{q_D}{H_D} \sum_{K\in\calT_h(D)}\norm{v_h-v_H}_{L^2(K)}^2 \Biggr].
\end{multline*}
Notice that the jumps $\llb v_H \rrb$ vanish for faces $F\in\calFhi(D)$.
Therefore, the approximation bound of Theorem~\ref{thm:dg_coarse_approximation} gives
\begin{multline}\label{eq:e_1_bound_0}
 E_1 \lesssim \max_{D\in\calT_H}\left[ \eta_D \frac{H_D}{q_D} \right]\max_{D\in\calT_H} \frac{H_D^2}{q_D^2} \normh{v_h}^2
 + \max_{D\in\calT_H}\left[ \eta_D \frac{H_D}{q_D} \max_{F\in\calF_h^i(D)}\frac{\tilde h_F^2}{\tilde{p}_F^4}\right] \absj{v_h}^2
\\ + \max_{D\in\calT_H}\left[ \eta_D \frac{q_D}{H_D} \right]\max_{D\in\calT_H} \frac{H_D^4}{q_D^4} \normh{v_h}^2,
\end{multline}
and thus it follows from \eqref{eq:c_h_bound} and \eqref{eq:c_p_bound} and coercivity of $a_h$ that
\begin{equation}\label{eq:e_1_bound_1}
 E_1 \lesssim \max_{D\in\calT_H}\left[ \frac{q_D}{H_D}  \max_{K\in\calT_h(D)} \frac{p_K^6}{h_K^3}\right] \max_{D\in\calT_H} \frac{H_D^4}{q_D^4} \, a_h(v_h,v_h).
\end{equation}
Remark that we have used the bounds $H_D/q_D \lesssim q_D / H_D \max_{D\in\calT_H} H_D^2/q_D^2$ and also $H_D/q_D \max_{F\in\calF_h^i(D)}\tilde h_F^2/\tilde{p}_F^4 \lesssim q_D/H_D \max_{D\in\calT_H} H_D^4/q_D^4$ in going from \eqref{eq:e_1_bound_0} to \eqref{eq:e_1_bound_1}. This is done because it is currently not possible to improve the last term in \eqref{eq:e_1_bound_0}, as a consequence of the nonlocal form of the bounds in Theorems~\ref{thm:dg_h2_approximation} and Theorem~\ref{thm:dg_coarse_approximation}.

The Cauchy--Schwarz inequality with a parameter and the symmetry of the sum over $i$, $j$, $j\neq i$, imply that
\begin{equation}\label{eq:E_k_bound}
\sum_{k=2}^5 E_k \lesssim 
\sum_{\substack{1\leq i\neq j\leq N\\F\in\calFhi \\ F\subset \p \Om_i \cap \p \Om_j}} \mu_F^{-1}\norm{\eval{D^2(v_h-v_H)}{\Om_i}}_{L^2(F)}^2 + \mu_F \norm{\eval{\nabla(v_h-v_H)}{\Om_j}}_{L^2(F)}^2.
\end{equation}
Since $\calT_S$ is conforming, each face $F$ may appear at most twice in the above sum, and thus the trace and inverse inequalities imply that 
\begin{equation}\label{eq:E_k_bound_2}
\sum_{\substack{1\leq i\neq j\leq N\\F\in\calFhi \\ F\subset \p \Om_i \cap \p \Om_j}} \mu_F^{-1}\norm{\eval{D^2(v_h-v_H)}{\Om_i}}_{L^2(F)}^2 \lesssim \sum_{K\in\calT_h}\norm{v_h-v_H}_{H^2(K)}^2 \lesssim \tilde{c}_0 a_h(v_h,v_h).
\end{equation}
Defining $\mu_D \coloneqq \max_{K\in\calT_h(D)} p_K^2/h_K$, we apply Lemma~\ref{lem:discrete_trace_inequality} componentwise to the gradient of $v_h-v_H$ to find that
\begin{multline}\label{eq:E_k_bound_3}
\sum_{\substack{i,\,j=1\\i\neq j}}^N \sum_{\substack{F\in\calFhi \\ F\subset \p \Om_i \cap \p \Om_j}} \mu_F \norm{\eval{\nabla(v_h-v_H)}{\Om_j}}_{L^2(F)}^2 \lesssim \sum_{D\in\calT_H} \mu_D \norm{\nabla (v_h-v_H)}_{L^2(\p D)}^2
\\ 
\lesssim \sum_{D\in\calT_H} \mu_D \Biggl[ \frac{H_D}{q_D} \sum_{K\in\calT_h(D)} \abs{v_h-v_H}_{H^2(K)}^2+\frac{H_D}{q_D} \sum_{F\in\calF_h^i(D)}\frac{\tilde{p}_F^2}{\tilde{h}_F} \norm{\llb \nabla v_h \rrb}_{L^2(F)}^2
\\  + \frac{q_D}{H_D} \sum_{K\in\calT_h(D)} \abs{v_h-v_H}_{H^1(K)}^2 \Biggr].
\end{multline}
It is important to observe that only terms involving interior faces of the mesh $\calT_h$ appear on the right-hand side of the above inequality, so for each $F\in\calF_h^i(D)$, we have
$\norm{\llb \nabla v_h \rrb}_{L^2(F)}^2 = \norm{\llb \nablaT v_h \rrb}_{L^2(F)}^2 + \norm{ \llb \nabla v_h \cdot n_F \rrb}_{L^2(F)}^2$. So, we deduce that
\begin{multline*}
\sum_{D\in\calT_H} \mu_D \norm{\nabla (v_h-v_H)}_{L^2(\p D)}^2
 \lesssim \max_{D\in\calT_H}\left[\mu_D \frac{H_D}{q_D} \right]  \norm{v_h-v_H}_{H^2(\Om;\calT_h)}^2 \\ + \max_{D\in\calT_H}\left[\mu_D \frac{H_D}{q_D} \right]  \absj{v_h}^2  + \max_{D\in\calT_H}\left[\mu_D\frac{q_D}{H_D}\right]\norm{v_h-v_H}_{H^1(\Om;\calT_h)}^2,
\end{multline*}
and thus Theorem~\ref{thm:dg_coarse_approximation} and coercivity of $a_h$ show that
\begin{equation}\label{eq:E_k_bound_4}
\sum_{D\in\calT_H} \mu_D \norm{\nabla (v_h-v_H)}_{L^2(\p D)}^2
  \lesssim \max_{D\in\calT_H}\left[ \frac{q_D}{H_D} \max_{K\in\calT_h(D)} \frac{p_K^2}{h_K} \right] \max_{D\in\calT_H} \frac{H_D^2}{q_D^2} \, a_h(v_h,v_h).
\end{equation}
Therefore, the inequalities \eqref{eq:E_k_bound}, \eqref{eq:E_k_bound_2} and \eqref{eq:E_k_bound_4} show that
\begin{equation}\label{eq:E_k_bound_5}
 \sum_{k=2}^5 E_k  \lesssim \max_{D\in\calT_H}\left[ \frac{q_D}{H_D} \max_{K\in\calT_h(D)} \frac{p_K^2}{h_K} \right] \max_{D\in\calT_H} \frac{H_D^2}{q_D^2} \, a_h(v_h,v_h).
\end{equation}
In summary, combining the inequalities \eqref{eq:stable_decomposition_2}, \eqref{eq:e_1_bound_1} and \eqref{eq:E_k_bound_5} implies that
\begin{equation}\label{eq:stable_decomposition_bound_final}
 \sum_{i=0}^N a_h^i(I_i v_i,I_i v_i) \lesssim \tilde{c}_0 \,a_h(v_h,v_h) + \sum_{k=1}^5 E_k  \lesssim \tilde{c}_0 \, a_h(v_h,v_h),
\end{equation}
which completes the proof of the stable decomposition property of Theorem~\ref{thm:stable_decomposition}.

The proof of Theorem~\ref{thm:stable_decomposition} completes the verification of Properties~\ref{prop:stable_decomposition}--\ref{prop:local_stability}, and thus gives the bound~\eqref{eq:condition_preconditioned} for the condition number of the preconditioned system.

\section{Numerical experiments}\label{sec:numexp}
We test the theoretical results of section~\ref{sec:preconditioners} and investigate the performance and competitiveness of the preconditioners in practical applications. Direct factorizations were used to form the coarse mesh and local solvers.

\subsection{Sharpness of the bound}\label{sec:numexp1}
Since the bound \eqref{eq:condition_preconditioned} is the first to be explicit in both coarse and fine mesh polynomial degrees, it is important to ascertain its sharpness.
Let $\Om=(0,1)^2$, and let the fixed meshes $\calT_H=\calT_S$ be obtained by a uniform subdivision of $\Om$ into $4$ squares, and let $\calT_h$ be obtained by uniform subdivision of $\Om$ into $16$ squares. We consider the sequence of spaces $\Vh$ of piecewise polynomials on $\calT_h$ with total degree $p$, where $p=2,\dots,12$, and the coarse spaces $\VH$ of piecewise polynomials on $\calT_H$ with total degree $q$, where $q=2,\dots, 6$.
We apply the additive Schwarz preconditioner defined in section~\ref{sec:preconditioners} to the bilinear form $a_h$ defined in \eqref{eq:a_h_definition}, where the penalty parameters are defined by $\cmu=\ceta=10$.
These choices are made to ensure that the resulting number of degrees of freedom is small, being at most equal to $1456$ in the case of $p=12$, thereby facilitating the accurate computation of the condition numbers $\kPA$ of the preconditioned matrix~$\bo P$.
The resulting condition numbers are given in Table~\ref{tab:pq_condition_numbers}, which shows that $\kPA$ is of order $1+p^6/q^3$, in agreement with the results of section~\ref{sec:preconditioners} and in particular with the bound \eqref{eq:predicted_condition_number}. This confirms that the predicted rates with respect to the polynomial degrees are optimal.
\cb{We further verify the sharpness of the bounds with respect to the parameters $H$ and $h$ in Table~\ref{tab:Hh_condition_numbers}, which presents the condition numbers for varying $h=2^{-m}$, $m=2,\dots,5$, and fixed $H=1/2$, and fixed $p=q=2$. It is seen that the predicted rate $\kPA$ is of order $H^3/h^3$ in agreement with the theory.}

\begin{table}[tb]
\begin{small}
\begin{center}
\begin{tabular}{| c | c c c c c | c |}
\hline
\vspace{-2.1ex} & & & & & & \\ 
  $\kPA$    & $q=2$ & $q=3$ & $q=4$ & $q=5$ & $q=6$  & $q$ rate \\ \hline
  \vspace{-2.1ex} & & & & & & \\ 
$p=2$      & $2.16\times10^1$ &&&&& \\
$p=3$      & $3.34\times10^2$ & $6.71\times10^1$&&&& \\
$p=4$      & $1.94\times10^3$ & $3.16\times10^2$&$1.35\times10^2$&&& \\
$p=5$      & $7.22\times10^3$ & $1.43\times10^3$&$4.11\times10^2$&$2.10\times10^2$&& \\
$p=6$      & $2.12\times10^4$ & $4.40\times10^3$&$1.31\times10^3$&$6.44\times10^2$&$3.03\times10^2$& $3.60$\\
$p=7$      & $5.31\times10^4$ & $1.10\times10^4$&$3.50\times10^3$&$1.70\times10^3$&$8.97\times10^2$& $3.35$\\
$p=8$      & $1.18\times10^5$ & $2.46\times10^4$&$7.91\times10^3$&$4.27\times10^3$&$2.10\times10^3$& $3.25$\\
$p=9$      & $2.38\times10^5$ & $4.88\times10^4$&$1.61\times10^4$&$8.68\times10^3$&$4.55\times10^3$& $3.10$\\
$p=10$     & $4.48\times10^5$ & $9.17\times10^4$&$3.00\times10^4$&$1.64\times10^4$&$8.86\times10^3$& $3.00$\\
$p=11$     & $7.92\times10^5$ & $1.61\times10^5$&$5.29\times10^4$&$2.90\times10^4$&$1.58\times10^4$& $2.97$\\
$p=12$     & $1.33\times10^6$ & $2.71\times10^5$&$8.89\times10^4$&$4.87\times10^4$&$2.66\times10^4$& $2.97$\\ \hline
$p$ rate   & $5.97$           &$5.94$    &$5.96$&$5.97$&$6.03$& \\ \hline
\end{tabular}
\caption{Dependence of the condition number $\kPA$ on the coarse and fine mesh polynomial degrees for the experiment of section~{\upshape\ref{sec:numexp1}}. The asymptotic rates are computed by regression on the last three entries of each column for $p$ and each row for $q$. It is found that $\kPA$ is of order  $1+p^6/q^3$, as predicted in section~{\upshape\ref{sec:preconditioners}}.}
\label{tab:pq_condition_numbers}
\end{center}
\end{small}
\end{table}

\begin{table}[tb]
\begin{small}
\begin{center}
\begin{tabular}{| c | c c c c | c |}
\hline
\vspace{-2.1ex} & & & & &  \\ 
  $\kPA$    & $h=1/4$ & $h=1/8$ & $h=1/16$ & $h=1/32$ & rate \\ \hline 
\vspace{-2.1ex} & & & & &  \\ 
 $H=1/2$ & $1.57\times 10^2$ & $1.19\times 10^3$ & $1.08\times 10^4$ & $8.90\times 10^4$ & $3.06$ \\ \hline
\end{tabular}
\caption{\cb{Dependence of the condition number $\kPA$ on the ratio of mesh sizes $H/h$, for fixed polynomial degrees $p$ and $q$. The asymptotic rates $\kPA$ is found to be of order $H^3/h^3$, in agreement with the bounds of section~{\upshape\ref{sec:preconditioners}}.}}
\label{tab:Hh_condition_numbers}
\end{center}
\end{small}
\end{table}

\subsection{Comparison with overlapping methods}\label{sec:numexp2}
In this section, we compare the efficiency of nonoverlapping methods with the closely related overlapping methods. It is found the methods achieve similar performances in terms of iteration counts, although nonoverlapping methods are often faster due to lower computational costs.

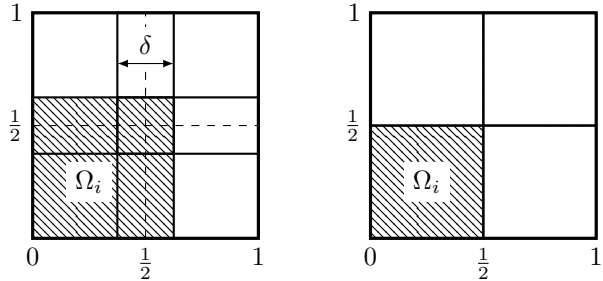
\begin{figure}
\begin{center}
\begin{tikzpicture}[scale=1.5,>=latex]
\node[below] at (-1,-1) {$0$};
\node[below] at (0,-1) {$\half$};
\node[below] at (1,-1) {$1$};
\node[left] at (-1,0) {$\half$};
\node[left] at (-1,1) {$1$};
\draw[very thick]  (-1,-1)  rectangle (1,1);
\draw[thick, pattern=north west lines] (-1,-1) rectangle   (0.25,0.25);
\draw[thick] (1,-1) rectangle (-0.25,0.25);
\draw[thick] (-1,1) rectangle (0.25,-0.25);
\draw[thick] (1,1) rectangle (-0.25,-0.25);
\draw[dashed] (-1,0)--(1,0);
\draw[dashed] (0,-1)--(0,1);
\draw node at (-0.5,-0.5) [fill=white] {$\Om_i$};
\draw[<->] (-0.25,0.55)-- node [above,fill=white] {$\delta$} (0.25,0.55) ;
\end{tikzpicture}
\qquad
\begin{tikzpicture}[scale=1.5]
\node[below] at (-1,-1) {$0$};
\node[below] at (0,-1) {$\half$};
\node[below] at (1,-1) {$1$};
\node[left] at (-1,0) {$\half$};
\node[left] at (-1,1) {$1$};
\draw[very thick]  (-1,-1)  rectangle (1,1);
\draw[thick, pattern=north west lines] (-1,-1) rectangle   (0.0,0.0);
\draw[thick] (1,-1) rectangle (-0.0,0.0);
\draw[thick] (-1,1) rectangle (0.0,-0.0);
\draw[thick] (1,1) rectangle (-0.0,-0.0);
\draw[dashed] (-1,0)--(1,0);
\draw[dashed] (0,-1)--(0,1);
\draw node at (-0.5,-0.5) [fill=white] {$\Om_i$};
\end{tikzpicture}
\caption{Overlapping and nonoverlapping decompositions of $\Om=(0,1)^2$ used in the experiment of section~{\upshape\ref{sec:numexp2}}. Four subdomains are used for both the overlapping and nonoverlapping methods, with the overlap size $\delta$ defined as the length shown above.}
\label{fig:overlap_dec}
\end{center}
\end{figure}

Let $\Om\coloneqq (0,1)^2$, and let $\calT_h$ be obtained by uniform subdivision of $\Om$ into squares of size $h=2^{-k}$, $k=3,\dots,8$. Let $\Vh$ consist of the space of polynomials of fixed partial degree $p=2$ on each element $K\in\calT_h$.
Consider the model problem: find $u_h\in \Vh$ such that $a_h(u_h,v_h)=\ell(v_h)$ for all $v_h\in\Vh$, where the linear functional $\ell_h$ is chosen so that the solution $u_h$ approximates the function $u(x,y)\coloneqq\mathrm{e}^{xy}\sin(\pi x)\sin(\pi y)$; specifically, we define $\ell(v_h) =  \sum_{K} \pair{ \Delta u}{\Delta v_h}_K$ for all $v_h\in\Vh$.
It can then be shown that $\norm{u-u_h}_{H^2(\Om;\calT_h)} \lesssim h^{p-1}$ \cite{Smears2013}. The penalty parameters are chosen so that $\mu_F = 10/\tilde{h}_F$ and $\eta_F = 10/\tilde{h}_F^3$.

\paragraph{Overlapping domain decomposition} Let $\delta\in(0,1)$ and let $\Om$ be divided into overlapping subdomains $\calT_S=\{\Om_i\}_{i=1}^4$, as shown in the left-hand side diagram of Figure~\ref{fig:overlap_dec}. This yields an overlapping decomposition of $\Om$ with overlap $\delta$; here, we use $\delta\in\{1/4,1/8,1/16\}$. Let $\calT_H$ be a coarse mesh consisting of a uniform subdivision of $\Om$ into $4$ squares, thus yielding the ratios $H/\delta \in \{2,4,8\}$, and let $\VH$ consist of the space of polynomials of fixed partial degree $q=2$ on each element $D\in\calT_H$. The local spaces $\Vh^i$ with associated solvers $a_h^i$, $1\leq i \leq 4$, are defined analogously to the nonoverlapping case, described in section~\ref{sec:preconditioners}. The additive Schwarz preconditioner is also defined analogously to section~\ref{sec:preconditioners}.

\paragraph{Nonoverlapping domain decomposition} The domain $\Om$ is partitioned into four subdomains $\calT_S=\{\Om_i\}_{i=1}^4$, as shown in the right-hand side diagram of Figure~\ref{fig:overlap_dec}. We consider three sequences of coarse meshes $\calT_H$, also obtained by uniform subdivision of $\Om$ into squares of size $H=2^{-m}$, $m=1,\dots,k-1$, so that $H/h\in\{2,4,8\}$. The nonoverlapping additive Schwarz preconditioner is defined as in section~\ref{sec:preconditioners}.

\paragraph{Results}
The implementations of the overlapping and nonoverlapping methods were the same, except for the required difference in handling the subdomains. Since the parallelizations of overlapping and nonoverlapping methods differ, our implementation was in serial in order to permit a more straightforward comparison.
Table~\ref{tab:iterations} gives the number of iterations required to reduce the residual norms by a factor of $10^{-6}$. The results for both methods are comparable to those in the literature: see for instance \cite{Antonietti2007,Antonietti2011,Brenner2005,Lasser2003}. 
Table~\ref{tab:iterations} also presents a representative sample of the CPU times required for the assembly of the preconditioner and the application of the preconditioned CG method. The assembly timing strictly includes the time spent on assembling and factorizing the coarse and local mesh solvers, whereas the solver time strictly includes the time spent on applying the preconditioned CG method. These timings are meant to provide only a relative comparison of the methods, with better absolute timings achievable by parallelization.

For the same iteration count, the nonoverlapping methods are generally faster in both assembly and solution. This advantage in efficiency is essentially the result of the smaller dimension of the subdomain solvers. The nonoverlapping method is also generally cheaper in terms of memory costs.
Our results show that both methods are efficient, with low iteration counts that remain bounded for fixed $H/\delta$ or $H/h$.
In both cases, the results are comparable to computational results from the literature~\cite{Antonietti2007,Feng2001}. \cb{The extension of the analysis for nonoverlapping preconditioners from this work to the case of overlapping preconditioners is an interesting problem for future work.}

\begin{table}
\begin{center}
\begin{small}
\begin{tabular}{ |c c | c c c |  c  c  c |}
\hline
& &  \multicolumn{6}{c|}{Iteration count}  \\ \hline
& &  \multicolumn{3}{c|}{Overlapping} & \multicolumn{3}{c|}{Nonoverlapping} \\
DoF & $h$  & $H=2\delta$ & $H=4\delta$ & $H=8\delta$  & $H=2h$ & $H=4h$  & $H=8h$  \\ \hline
144 & 1/4 & & & & 20 &   &  \\
576 & 1/8 & 18 & & & 22  & 29  &  \\
2304 & 1/16  & 18 & 24 & & 22  & 30  & 43 \\
9216 & 1/32  & 18 & 25 & 37 & 20  & 32  & 52 \\
36864 & 1/64 & 18 & 25 & 41 & 18 & 30 & 50  \\
147456 & 1/128 & 18 & 26 & 41 & 17 & 27  & 48 \\
589824 & 1/256 & 18 & 26 & 42 & 17 & 25  & 40 \\
 \hline 
 \multicolumn{8}{c}{~} \\
 \hline
 & &  \multicolumn{6}{c|}{Timing}  \\ \hline
& &  \multicolumn{3}{c|}{Overlapping} & \multicolumn{3}{c|}{Nonoverlapping} \\
\multicolumn{2}{|c|}{$h=1/128$} & $H=2\delta$ & $H=4\delta$ & $H=8\delta$  & $H=2h$ & $H=4h$  &  $H=8h$  \\ \hline
\multicolumn{2}{|c|}{Assembly time }& 18.6s & 14.5s & 13.0s & 14.0s & 11.9s & 11.6s  \\
\multicolumn{2}{|c|}{Solver time } & 8.39s & 9.56s & 13.3s & 6.51s & 8.62s & 14.4s \\ \hline
\end{tabular}
\caption{Number of preconditioned CG iterations required to reduce the residual norm by a factor of $10^{-6}$ for overlapping and nonoverlapping methods, in the experiment of section~{\upshape\ref{sec:numexp2}}, along with sample timings for assembly and timings of the preconditioned CG algorithm. The methods yield similar iteration counts for similar ratios of $H/\delta$ or $H/h$, but the nonoverlapping method is faster to assemble and apply, as a result of the smaller number of degrees of freedom in the local solvers.}
\label{tab:iterations}
\end{small}
\end{center}
\end{table}

\subsection{Application to HJB equations}\label{sec:numexp3}
We will now consider applications of the preconditioning methods to problems of practical interest, namely fully nonlinear HJB equations. As explained above, this introduces several challenges, such as nonsymmetric linear systems that appear in the semismooth Newton method. Nevertheless, it is found that nonoverlapping methods in particular remain robust and lead to efficient solvers for these problems for $h$-version methods.
The example presented here is closely related to the one from \cite[Section~9.1]{Smears2014}. 
Consider the boundary-value problem
\begin{equation}\label{eq:HJB_PDE2}
\begin{aligned}
\sup_{\a\in\Ld}\left[ L^\a u - f^\a \right] &=0 & &\text{in }\Om,\\
u &= 0  & &\text{on }\DO,
\end{aligned}
\end{equation}
where $\Om=(0,1)^2$, $\Ld\coloneqq [0,\pi/3]\times\mathrm{SO}(2)$, and where $L^\a v \coloneqq a^\a\colon D^2v$, with
\begin{equation}
\begin{aligned}
a^\a &\coloneqq \frac{1}{2} R\begin{pmatrix}
1+\sin^2\theta & \sin\theta\,\cos\theta \\
\sin\theta\,\cos\theta & \cos^2\theta
\end{pmatrix}
R^\top,  & \a &=(\theta,R)\in\Ld.
\end{aligned}
\end{equation}
The source terms $f^\a$, $\a\in\Ld$, are chosen so that the exact solution is given by $u(x,y)= \mathrm{e}^{xy}\sin(\pi x)\sin(\pi y)$, whilst yielding large variations in the values of $\a$ that attain the supremum in \eqref{eq:HJB_PDE2}. 
As explained in \cite{Smears2014}, a key challenge in this example is that the diffusion coefficient $a^\a$ is highly anisotropic for $\theta$ near $\pi/3$, and the rotation matrices $R$ may lead to large variations in the resulting diffusions across the domain and between Newton steps.
As a result, significant anisotropic variations in the resulting linearizations are encountered in the application of the semismooth Newton method.

\begin{table}
\begin{center}
\begin{small}
\begin{tabular}{ |c c | c c c |  c  c  c |}
\hline
&   \multicolumn{7}{c|}{Average GMRES iterations (Newton steps)}  \\ \hline
& &  \multicolumn{3}{c|}{4 Subdomains} & \multicolumn{3}{c|}{16 Subdomains} \\
DoF & $h$  & $H=2h$ & $H=4h$ & $H=8h$  & $H=2h$ & $H=4h$  & $H=8h$  \\ \hline
144 & 1/4       & 14.3 (6)  &        	&          	&        	&   		&   \\
576 & 1/8       & 15.2 (5)  & 18.8 (5)  &      		& 17.8 (5) 	&   		&	\\
2304 & 1/16     & 15.4 (5)  & 20.0 (5) 	& 26.8 (5)  & 18.0 (5) 	& 25.0 (5) 	& 	\\
9216 & 1/32     & 16.3 (6)  & 19.7 (6) 	& 29.5 (6) 	& 17.3 (6) 	& 24.0 (6) 	& 36.5 (6)	\\
36864 & 1/64    & 16.0 (6)  & 18.3 (6) 	& 26.3 (6) 	& 17.2 (6) 	& 22.0 (6) 	& 32.8 (6)	\\
147456 & 1/128  & 16.3 (6)  & 18.3 (6) 	& 23.0 (6) 	& 17.0 (6) 	& 19.8 (6) 	& 28.0 (6)	\\
\hline 
\end{tabular}
\caption{Average number of GMRES iterations per Newton step, with total number of Newton steps in parentheses, for the problem of section~{\upshape\ref{sec:numexp3}} with both $4$ and $16$ subdomains.}
\label{tab:GMRES_averages}
\end{small}
\end{center}
\end{table}

\begin{table}
\begin{center}
\begin{small}
\begin{tabular}{ |c c | c c c c |}
\hline
\multicolumn{6}{|c|}{Average GMRES iterations (Newton steps)} \\ \hline
DoF &  $p$  & $q=3$ & $q=4$ & $q=5$  & $q=6 $\\ \hline
16384 & $3$ & 22.0 (6) &  &  &    \\ 
25600 & $4$ & 25.8 (6) & 20.8 (6) & &  \\
36864 & $5$ & 28.8 (6) & 22.0 (6) & 20.8 (6) &  \\
50176 & $6$ & 31.5 (6) & 23.3 (6) & 22.2 (6) & 20.8 (6)  \\
65536 & $7$ & 35.2 (6) & 24.0 (6) & 23.5 (6) & 21.2 (6) \\
82944 & $8$ & 37.0 (6) & 25.2 (6) & 25.0 (6) & 21.8 (6) \\ \hline
\end{tabular}
\caption{\cb{Average number of GMRES iterations per Newton step, with total number of Newton steps in parentheses, for the problem of section~{\upshape\ref{sec:numexp3}} with varying polynomial degrees $p$ and $q$, with fixed $h=1/32$, $H=2h$, 1024 elements and $256$ subdomains.}}
\label{tab:GMRES_averages_pq}
\end{small}
\end{center}
\end{table}

\begin{table}
\begin{center}
\begin{small}
\begin{tabular}{ | c | c c c c c c|}
\hline
 \multicolumn{7}{|c|}{ Average GMRES iterations (Newton steps) } \\ \hline 
 & & & & & & \vspace{-2.1ex} \\
Subdomains & $h=1/4$ & $h=1/8$ & $h=1/16$ & $h=1/32$ & $h=1/64$ & $h=1/128$ \\ \hline 
4 & 17.2 (5) & 17.3 (6) & 17.7 (6) & 17.7 (6) & 17.7 (6) & 17.3 (6) \\
16  &  & 19.2 (6) & 18.8 (6) & 18.5 (6) & 18.5 (6) & 18.2 (6) \\
64 &  &  & 20 (6) & 19.5 (6) & 19.5 (6) & 19.3 (6) \\
256 & & & & 20.8 (6) & 20.5 (6) & 20.3 (6) \\ \hline 
\end{tabular}
\caption{\cb{Average number of GMRES iterations per Newton step required for a relative residual norm tolerance $10^{-4}$, with total number of Newton steps in parentheses, for varying numbers of subdomains, using $H=2h$ and $p=4$.}}
\label{tab:GMRES_average_subdomains}
\end{small}
\end{center}
\end{table}

\begin{table}
\begin{center}
\begin{small}
\begin{tabular}{ | c | c c c c |}
\hline
$h$ & $p=2$ & $p=3$ & $p=4$  & $p=5$   \\ \hline
1/4   & 18 	& 21 &  21 & 22 \\
1/8   & 19  & 20 &	19 & 20 \\
1/16  & 19  & 19 &  19 & 19 \\
1/32  & 18  & 19 & 	17 & 18 \\
1/64  & 17  & 19 &  16 & 17 \\
\hline 
\end{tabular}
\caption{Number of GMRES iterations at the first Newton step required for a relative residual norm tolerance of $10^{-6}$, for various polynomial degrees $2\leq p \leq 5$, using $H=2h$ and $4$ subdomains. The results are better than theoretical predictions, see Remark~{\upshape\ref{rem:GMRES_descriptive}}.}
\label{tab:GMRES_polynomial}
\end{small}
\end{center}
\end{table}

\begin{table}
\begin{center}
\begin{small}
\begin{tabular}{ | c | c c c c c c|}
\hline
 & \multicolumn{6}{c|}{ $h$ } \\ 
Subdomains & 1/4 & 1/8 & 1/16 & 1/32 & 1/64 & 1/128 \\ \hline 
4 & 18 & 19 & 19 & 18 & 17 & 15 \\
16  &  & 22 & 20 & 19 & 18 & 17 \\
64 &  &  & 21 & 19 & 18 & 17 \\
256 & & & & 21 & 19 & 18 \\ \hline 
\end{tabular}
\caption{Number of GMRES iterations on the first Newton step required for a relative residual norm tolerance $10^{-6}$, for varying numbers of subdomains, using $H=2h$ and $p=2$.}
\label{tab:GMRES_n_subdomains}
\end{small}
\end{center}
\end{table}

The numerical scheme \eqref{eq:HJB_scheme} is applied on a sequence of fine meshes $\calT_h$ obtained by uniform subdivision of $\Om$ into squares of size $h=2^{-k}$, $k=3,\dots,7$, with polynomial degrees $2\leq p \leq 5$.
Each iteration of the semismooth Newton method for solving \eqref{eq:HJB_scheme} leads to a nonsymmetric but positive definite linear system \cite{Smears2014}, which we solve using the GMRES method \eqref{eq:GMRES_left_prec} implemented as suggested in \cite{Saad2003}. The nonoverlapping preconditioners are based on the bilinear form $a_h$, using between $4$ and $256$ regular subdomains and $q=p$.

To study the overall performance of the preconditioners, we computed the average number of GMRES iterations per Newton step required to reduce the residual norm $\norm{\bo r_k}_{\bo P^{-1}}$ below a tolerance of $10^{-6}$ or a relative tolerance of $10^{-4}$. Convergence of the Newton method was determined by requiring a step-increment $L^2$-norm below $10^{-6}$. These tolerances were chosen to give a good balance between the different sources of error originating from discretization, linearization and algebraic solvers.
The corresponding results are given in Table~\ref{tab:GMRES_averages}, showing the effectiveness of the preconditioners and their robustness with respect to the anisotropy of the diffusion term. \cb{Tables~\ref{tab:GMRES_averages_pq} and \ref{tab:GMRES_polynomial} shows the iteration counts for varying choices of the polynomial degrees. Tables~\ref{tab:GMRES_average_subdomains} and~\ref{tab:GMRES_n_subdomains} shows that the iteration counts are not affected by the number of subdomains.}
We point out that these iteration counts are comparable to those obtained by Lasser and Toselli in \cite{Lasser2003} for nonsymmetric $H^1$-type problems originating from advection-diffusion-reaction equations. In particular, for moderate polynomial degrees, the preconditioners are found to be efficient and robust under $h$-refinement.

Overall, these results show that nonoverlapping preconditioners are robust and efficient when confronted with the anisotropy, lack of symmetry and nonlinearity of this problem.

\section{Conclusion}\label{sec:conclusion}
Original approximation results for discontinuous finite element spaces lead to optimal order spectral bounds for nonoverlapping domain decomposition preconditioners in $H^2$-norms. In the case of $h$-refinement, we have shown the robustness, efficiency and competitiveness of these preconditioning methods in applications to the nonsymmetric systems arising from fully nonlinear HJB equations.

\end{document}